\documentclass[11pts]{article}
\usepackage[top=1in, bottom=1in, left=1in, right=1in]{geometry}
\usepackage{fancyhdr}
\usepackage{amssymb}
\usepackage{bbm}
\usepackage{amsmath}
\usepackage{amsthm}
\usepackage[english]{babel}
\usepackage[latin1]{inputenc}
\usepackage{amsmath}
\usepackage{graphics}
\usepackage{latexsym}

\usepackage{amsfonts}
\numberwithin{equation}{section}

\def \esp {[0,T]\times \R^k}

\def \tp{^{\top}}

\def \R {{\bf R}}
 \def \triangleq {=}

\def\G{\Gamma}
\def \mps {\mapsto}

\def \ind{1\!\!1}

\def \ed {\end{document}}
\def \tx {(t,x)\in \esp}

\def \ij {(i,j)\in A^1\times A^2}

\newcommand{\ba}{\begin{array}}
\newcommand{\ea}{\end{array}}

\def \eps {\epsilon}
\def \d {\delta}

\def \ms {\medskip}
\def \bs {\bigskip}
\def \cF {{\cal F}}

\def\no{\noindent}
\def \qq {\qquad}

\def \bal{\begin{array}{l}}
\def \ea{\end{array}}
\def \q {\quad}
\def \beq {\begin{eqnarray}}
\def \eeq {\end{eqnarray}}

\def \xt {X^{t,x}}
\def \cP {\mathbf{P}}
\def \E {\mathbf{E}}
\def \cS {{\cal S}}
\def \cL {{\cal L}}
\def \cH {{\cal H}}
\def \cA {{\cal A}}

\def \ms {\medskip}
\def \bs {\bigskip}
\def \cF {{\cal F}}

\def \qq {\qquad}

\def \bal{\begin{array}{l}}
\def \ea{\end{array}}
\def \q {\quad}

\def \xt {X^{t,x}}
\def \E {\mathbf{E}}
\def \cF {\mathcal{F}}
\def \cP {\mathcal{P}}

\def \cS {{\cal S}}

\def \cL {{\cal L}}
\def \cH {{\cal H}}
\def \cA {{\cal A}}

\def\b{\beta}
\def\g{\gamma}
\def\d{\delta}

\def\l{\lambda}

\def\mb{\mbox}

\def\cA{{\cal A}}

\def\cC{{\cal C}}

\def\cF{{\cal F}}

\def\cH{{\cal H}}

\def\cL{{\cal L}}

\def\cP{{\cal P}}

\def\cS{{\cal S}}

\def \ind{1\!\!1}

\def \ed {\end{document}}
\def \tx {(t,x)\in \esp}

\def \kl {(k,l)\in A^1\times A^2}

\def \E{\mathbb{E}}

\newcommand{\be}{\begin{equation}}
\newcommand{\ee}{\end{equation}}
\newcommand{\bea}{\begin{eqnarray}}
\newcommand{\eea}{\end{eqnarray}}
\def \rw {\rightarrow}

\newtheorem{thm}{Theorem}[section]
\newtheorem{rem}{Remark}[section]
\newtheorem{propo}{Proposition}[section]
\newtheorem{defi}{Definition}[section]
\newtheorem{lem}{Lemma}[section]
\newtheorem{cor}{Corollary}[section]

\def \id1x {I_{ij}^{1,\d}(t,x,\phi)}
\def \id2x {I_{ij}^{2,\d}(t,x,\phi)}
\def \id1fx {\int_{|e|\leq
\delta}(\phi(t,x+\beta(x,e))-\phi(t,x))\gamma^{ij}(x,e)n(de)}
\def \id2fx{
\int_{|e|\geq
\delta}(\phi(t,x+\beta(x,e))-\phi(t,x))\gamma^{ij}(x,e)n(de)}
\def \ijg {(i,j)\in {\Gamma}}

\def \pin {\mbox{$\int_E$}}
\def \ijg {(i,j)\in \Gamma}
\def \klg {(k,l)\in \Gamma}

\begin{document}
\bibliographystyle{unsrt}
\title{Viscosity solutions of systems of variational inequalities with
interconnected bilateral obstacles of non-local type.}
\author{Said Hamad\`ene\footnote{
Universit\'e du Maine, LMM, Avenue Olivier Messiaen, 72085 Le Mans,
Cedex 9, France. Email: hamadene@univ-lemans.fr } \,\, and Xuzhe
Zhao\footnote{ Applied Mathematics Department School of Finance University of Foreign Studies, Guangzhou 510420, P.R.China. E-mail: sosmall129@hotmail.com} }
 \date{\today}
\maketitle
\begin{abstract}
In this paper, we study systems of nonlinear second-order
variational inequalities with interconnected bilateral obstacles
with non-local terms. They are of min-max and max-min types and
related to a multiple modes zero-sum switching  game in the
jump-diffusion model. Using systems of penalized reflected backward
SDEs with jumps and unilateral interconnected obstacles, and their associated deterministic functions, we
construct for each system a continuous viscosity solution which is
unique in the class of functions with polynomial growth.
\medskip

\medskip

\end{abstract}
\setcounter{table}{1}
\section{\Large\bf Introduction}
During the last decade optimal stochastic switching problems have
attracted a lot of research activity (see e.g.
\cite{eliekharroubichassegneux,dhp,  eliekharroubi, hasri,ibtissam,hj,
hamadenemorlais, hz, hutang, pham, yongtang} and the references therein) in
connection with their various applications especially in the economic
and finance spheres, such as energy, etc. Comparatively, switching
games, of zero-sum or nonzero-sum types, have been less considered
even though there are some works in this field including
\cite{dhmz, hu,koike, lyama, tang}. In these latter articles, the
Hamilton-Jacobi-Bellmans-Isaacs PDE, which is of min-max or max-min
type, associated with the zero-sum switching stochastic game is
studied from the point of view of viscosity solution theory. The
probabilistic version of those works is considered in
\cite{dhmz,hu} where it is shown that the BSDE system associated
with the zero-sum game has a solution. In \cite{dhmz}, uniqueness of the solution, 
which is an involved question, is proved as well. The issue of existence of a value or a saddle-point for
the game is also addressed in \cite{dhmz}, where it is shown that the game
has a saddle-point when the switching costs and utilities are decoupled. This
existence is deeply related to the optimal policy of a standard optimal switching problem. The general case still
open. \ms

Except articles \cite{zhao, ibtissam}, all the previous works deal with the
case of Brownian noise. In \cite{zhao}, the framework where the
noise is driven by a L\'evy process is studied in detail. The main
motivation is that models which include Poisson jumps have the
feature to be more realistic since they capture non-predictable
events, e.g. in the energy market, jumps of the prices due to sudden
weather changes, etc. Therefore the main objective of this work is the extension to
the model with jumps of the paper  \cite{dhmz}, where the authors have studied systems of
variational inequalities with interconnected lower and upper
obstacles, which arise as the Hamilton-Jacobi-Bellman-Isaacs equation
in a multiple modes switching game between two players in the
framework without jumps. Precisely we consider the
following system of non-local variational inequalities or integral-partial differential equations (IPDEs for short):  For every pair $(i,j)$ in the finite set of modes
$A^1\times A^2$, \ms
 \be
\label{1SIPDE}\left\{\begin{array}{ll}
\min\{(v^{ij}-L^{ij}[(v^{kl})_{(k,l)\in A^1\times A^2}])(t,x),
\max\{(v^{ij}-U^{ij}[(v^{kl})_{(k,l)\in A^1\times A^2}])
(t,x),\\-\partial_t v^{ij}(t,x)-\cL
v^{ij}(t,x)-g^{ij}(t,x,(v^{kl}(t,x))_{(k,l)\in A^1\times
A^2},\sigma(t,x)^\top D_x
v^{ij}(t,x),I_{ij}(t,x,v^{ij}))\}\}=0;\\
v^{ij}(T,x)=h^{ij}(x)\end{array}\right.\ee where, for any $(t,x)\in
[0,T]\times \R^k$,\be \label{1SIPDEnot}\begin{array} {l} \mbox{a) }
\cL\phi(t,x):=b(t,x)^\top D_x\phi(t,x)+\frac{1}{2}Tr[\sigma
\sigma^\top(t,x)D_{xx}^2\phi(t,x)]+\\\qq\qq\qq\qq\qq\qq\qq\pin(\phi(t,x+\beta(x,e))-
\phi(t,x)-D_x\phi(t,x)\beta(x,e))n(de);
\\\mbox{b) }I_{ij}(t,x,\phi)=\pin(\phi(t,x+\beta(x,e))-\phi(t,x))\gamma^{ij}(x,e)n(de);\\
\mbox{c) } L^{ij}[(v^{kl})_{(k,l)\in A^1\times
A^2}])(t,x):=\max\limits_{k\neq i}(v^{kj}-\underline{g}_{ik})(t,x)) \mbox{ and }
\\\qq\qq\qq\qq\qq\qq ~U^{ij}[(v^{kl})_{(k,l)\in A^1\times
A^2}])(t,x):=\min\limits_{l\neq j}(v^{il}+\overline{g}_{jl})(t,x)).\end{array}\ee The function
$\underline{g}_{ik}$ (resp. $\overline{g}_{jl}$) stands for the
switching cost of the maximizer (resp. minimizer) when she makes the
decision to switch from mode $i$ to mode $k$ (resp. mode $j$ to mode
$l$) while the function $g^{ij}$ (resp. $h^{ij}$) is the
instantaneous (resp. terminal) payoff when the maximizer (resp.
minimizer) chooses mode $i$ (resp. $j$). The non-local terms which
appear in (\ref{1SIPDE}) and given in a), b) above stem from the
jumps of the dynamics of the system which is of jump-diffusion type
(see (\ref{cmSDE}) below). Finally note that the obstacles in (\ref{1SIPDE}) depend on the solution. 

In this paper we show that system (\ref{1SIPDE}) has a continuous
solution in viscosity sense which is moreover unique in the class of
functions which have polynomial growth. As a by-product we obtain
the same conclusion for the max-min system (\ref{maxSIPDE}). 
Our work should be seen as a starting point for future research in this field (e.g. improvement of the results, numerics, etc.).

This paper is organized as follows: In Section 2, we fix the
notations, assumptions, definitions and set up accurately the
problem. In Section 3, we prove a comparison result between the
subsolutions and supersolutions of system (\ref{1SIPDE}) when they
have polynomial growth. As an immediate consequence, the solutions of
(\ref{1SIPDE}) with polynomial growth is necessarily continuous and unique. In
Section 4, we introduce systems of integral-partial differential equations 
with lower (resp. upper)
interconnected obstacles which are obtained by the penalization of
the upper (resp. lower) obstacles of system (\ref{1SIPDE}). They are approximating schemes 
for (\ref{1SIPDE}) and max-min system respectively. We
highlight some of their properties in making the connection with
systems of reflected BSDEs with lower (resp. upper) obstacles.  Later
on we show that system (\ref{1SIPDE}) has a subsolution and a
supersolution as well. Finally by Perron's method we show that
it has a unique solution. As a by product we show also that the
max-min system has a unique solution. At the end, there is an
Appendix, where we give another definition of the viscosity solution
of system (\ref{1SIPDE}) which uses ``local'' maxima and minima and which is inspired by the work
by Barles-Imbert in \cite{barlesimbert}.
\section{ Preliminaries} Let $(\Omega,\cF,(\cF_t)_{t\geq 0},P)$ be a stochastic basis such
that $\cF_0$ contains all $P$-null sets of $\cF$, and
$\cF_{t^+}\triangleq\bigcap\limits_{\varepsilon>0}\cF_{t+\varepsilon}:=\cF_{t}$,
$t\geq 0$. We suppose that the filtration is generated by the
following two
mutually independent processes:\\
- a $d$-dimensional standard Brownian motion $(B_t)_{t\geq 0}$ ;\\
- a Poisson random measure $N$ on $\R_+\times E$, where $E\triangleq
\R^l-\{0\}$ ($l\geq 1$ fixed) is equipped with its Borel field
$\mathcal{B}_E$, with compensator $\nu(dt,de)=dtn(de)$, such that
$\{\hat{N}((0,t]\times A)=(N-\nu)((0,t]\times A)\}_{0\leq t\leq T}$
is an $\cF_t$-martingale for all $A\in\mathcal{B}_E$ satisfying
$n(A)<\infty$. The measure $n$ is assumed to be $\sigma$-finite on $(E,\mathcal{B}_E)$  and integrates $(1\wedge
|e|^2)_{e\in E}$. \ms

Let $T$ be a fixed positive constant and let $A^1$ (resp.  $A^2$)
denote the set of switching modes for player 1 (resp. player 2)
whose cardinal is $m_1$ (resp. $m_2$). The set $ A^1\times A^2$ will
be sometimes simply denoted by $\Gamma$. For $\ijg$, we set
$A^1_i:=A^1-\{i\}$, $A^2_j:=A^2-\{j\}$ and
$\G^{-(i,j)}=\G-\{(i,j)\}$. Next, for $\vec{y}=(y^{kl})_{(k,l)\in
A^1\times A^2}\in \R^{m_1\times m_2}$ and $y_1\in \R$, we denote by
$[\vec{y}^{ij},y_1]$ the matrix obtained from
$\vec{y}$ by replacing the element $y^{ij}$ with $y_1$.

A function $\Phi :(t,x)\in [0,T]\times \R^k \mps\Phi(t,x)\in \R$ is
called of polynomial growth if there exist two non-negative real
constant $C$ and $\gamma$ such that for any $\tx$,
$$|\Phi(t,x)|\leq C(1+|x|^\gamma).$$Hereafter, this class of
functions is denoted by $\Pi_g$.\ms

We now define the probabilistic tools and sets we need later. Let:
\ms

\noindent (i) $\cP$ be the $\sigma$-algebra of $\cF_t$-predictable subsets of $\Omega\times[0,T]$;\\
\noindent (ii) $\cH^{2}:= \{\varphi:=(\varphi_t)_{t\leq T}$ is an
$\R^d$-valued, ${\cal F}_t$-progressively measurable process s.t.
$\|\varphi\|^2_{\cH^2}:=\E(\int^T_0{\vert {\varphi_t}\vert}^2dt)<\infty\}$
;\\
\noindent (iii) $\cS^2:= \{\xi:=(\xi_t)_{t\leq T}$ is an
$\R$-valued, ${\cal F}_t$-adapted RCLL process  s.t.
$\|\xi\|^2_{\cS^2}:=\E[\sup_{0\leq t\leq T}{\vert
{\xi_t}\vert}^2]<\infty\}$ ; $\cA^2$ is the subspace
of $\cS^2$ of continuous non-decreasing processes null at $t=0$ ; \\
\noindent (iv) $\cH^2(\tilde{N}):= \{U:\Omega\times[0,T]\times
E\rightarrow \R, \cP\otimes\mathcal{B}_E$-measurable and s.t.
$\|U\|^2_{\cH^2(\tilde{N})}:=\E(\int^T_0\int_E|U_t(e)|^2n(de)dt)<\infty\}
$.
\medskip

The main objective of this paper is to investigate the problem of existence and uniqueness of a viscosity
solutions $\vec{v}(t,x):=(v^{ij}(t,x))_{(i,j)\in A^1\times A^2}$ of the following system of non-local variational inequalities (SVI in short) or IPDEs with upper and lower interconnected obstacles: $\forall (i,j)\in
A^1\times A^2$,\be \label{SIPDE}\left\{
    \begin{array}{ll}
\min\{(v^{ij}-L^{ij}[\vec{v}])(t,x) ; \max\{(v^{ij}-U^{ij}[\vec{v}])(t,x);\\
-\partial_t v^{ij}(t,x)-\cL
v^{ij}(t,x)-g^{ij}(t,x,(v^{kl}(t,x))_{(k,l)\in
A^1\times A^2},\sigma(t,x)^\top D_x v^{ij}(t,x),I_{ij}(t,x,v^{ij}))\}\}=0\,;\\\\
     v^{ij}(T,x)=h^{ij}(x),
    \end{array}
    \right.\ee
where for any $\ijg$, $(t,x)\in [0,T]\times \R^k$ and $\phi\in
\cC^{1,2}$, $L^{ij}[\vec{v}]$, $U^{ij}[\vec{v}]$, ${\cal L}v^{ij}(t,x)$ and $I_{{ij}}(t,x,\phi)$ are given in 
(\ref{1SIPDEnot}). The functions $\underline{g}_{ik}, \overline{g}_{ik}$, $\beta$ and
$\gamma^{ij}$ are given and will be specified more later.

Next for $\d>0$, $(t,x)\in \esp$, $\zeta \in \R^k$, $\phi$ a
$\cC^{1,2}$-function and $\ijg$, let us set:$$\begin{array}{l} (a)\,
I^1_\delta(t,x,\phi)=\int_{|e|\leq
\delta}(\phi(t,x+\beta(x,e))-\phi(t,x)-D_x\phi(t,x)\beta(x,e))n(de);\\
(b)\, I^2_\delta(t,x,\zeta ,\phi)=\int_{|e|\geq\delta
}(\phi(t,x+\beta(x,e))-\phi(t,x)-\zeta \beta(x,e))n(de);\\
(c) \,I_{ij}^{1,\d}(t,x,\phi)=\int_{|e|\leq
\delta}(\phi(t,x+\beta(x,e))-\phi(t,x))\gamma^{ij}(x,e)n(de);\\
(d)\, I_{ij}^{2,\d}(t,x,\phi)=\int_{|e|\geq\delta}(\phi(t,x+\beta(x,e))-\phi(t,x))\gamma^{ij}(x,e)n(de);\\
(e)\,\,\bar \cL\phi(t,x):=b(t,x)^\top
D_x\phi(t,x)+\frac{1}{2}Tr[\sigma\sigma^\top(t,x)D_{xx}^2\phi(t,x)].
\end{array}$$
Note that for any $\d>0$ and $\ij$,
$$
I(t,x,\phi)=I^1_\delta(t,x,\phi)+I^2_\delta(t,x,D_x\phi,\phi)
\mbox{ and }I_{{ij}}(t,x,\phi)=I_{ij}^{1,\delta}(t,x,\phi)+I_{ij}^{2,\delta}(t,x,\phi).
$$
Next the following assumptions will be in force throughout the rest of
this paper.\\
\noindent\underline {\bf(A0)}:
    \ms

\no (i) The function $b(t,x)$ (resp. $\sigma(t,x)$): $[0,T]\times
\R^k\rightarrow \R^k$ (resp. $\R^{k\times d}$) is jointly continuous
in $(t,x)$ and Lipschitz continuous w.r.t. $x$, meaning that there
exists a non-negative constant $C$ such that for any $(t,x,x')\in
[0,T]\times \R^{k+k}$ we
have:$$|\sigma(t,x)-\sigma(t,x')|+|b(t,x)-b(t,x')|\leq C|x-x'|.$$
Combining this property with continuity one deduces that $b$ and
$\sigma$ are of linear growth w.r.t. $x$, i.e.,
$$|b(t,x)|+|\sigma(t,x)|\leq C(1+|x|).$$ (ii) The function $\beta:
\R^k\times E\rightarrow \R^k$ is measurable, and
such that for some real $K$,  $$|\beta(x,e)|\leq K(1\wedge|e|)\mbox{
and }|\beta(x,e)-\beta(x',e)|\leq K|x-x'|(1\wedge|e|), \,\,\forall
e\in E \mbox{ and }x,x'\in \R^k.$$ \noindent\underline {\bf(A1)}:
\ms

\no For any $\ijg$, the function $g^{ij}:\,\,(t,x,\vec{y},z,q)\in [0,T]\times \R^{k+m_1
\times m_2+d+1}\longmapsto g^{ij}(t,x,\vec{y},z,q) \in \R$ verifies:\\
(i) it is continuous in $(t,x)$ uniformly w.r.t. the other variables
$(\vec{y},z,q)$ and for any $(t,x)$  the mapping
$(t,x)\mapsto g^{ij}(t,x,0,0,0)$ is of polynomial growth ;\\
(ii) it satisfies the standard hypothesis of Lipschitz continuity
w.r.t. the variables $(\vec{y},z,q)$, i.e. for any $(t,x)\in
[0,T]\times \R^k$, $(\vec{y}_1,\vec{y}_2)\in (\R^{m_1\times m_2})^2$
and $(z_1,z_2)\in (\R^{d})^2$, $q_1,q_2\in \R$, it holds
$$|g^{ij}(t,x,\vec{y}_1,z_1,q_1)-g^{ij}(t,x,\vec{y}_2,z_2,q_2)|\leq
C(|\vec{y}_1-\vec{y}_2|+|z_1-z_2|+|q_1-q_2|),$$ where $|\vec{y}|$
stands for the standard Euclidean norm
of $\vec{y}$ in $\R^{m_1\times m_2}$; \\
(iii) the mapping $q\mapsto g^{ij}(t,x,y,z,q)$ is non-decreasing,
for all fixed $(t,x,y,z)\in[0,T]\times \R^{k+m_1\times m_2+d}$.\qed \\
\ms

\noindent Next for any $\ijg$, the function $\gamma^{ij}:\R^k\times E
\rightarrow \R$ verifies for some constant $C$ :
\be\label{condgamma}
\begin{array}{l}
\mb{(a)}\,\, |\gamma^{ij}(x,e)-\gamma^{ij}(x',e)|\leq C|x-x'|(1\wedge|e|),~~x,x'\in
\R^k \mbox{ and } e\in E\, ;\\
\mb{(b) } \q 0\leq\gamma^{ij}(x,e)\leq C(1\wedge|e|),~~x\in
\R^k\mbox{ and } e\in E. \ea \ee Finally let us define functions
$(f^{ij})_{\ij}$, on $[0,T]\times \R^{k+m_1\times m_2+d}\times {\cal
L}_\R^2(E,{\cal B}_E,n)$, as follows:\be \label{deffij}f^{ij}(t,x,\vec
y,z,u):=g^{ij}(t,x,\vec y,z,\pin u(e)\gamma^{ij}
(x,e)n(de)).\ee

\noindent\underline {\bf(A2)}: \underline {Monotonicity}: For any
$\ijg$ and any $(k,l)\neq(i,j)$, the mapping $y^{kl}\rightarrow
g^{ij}(t,x,\vec{y},z,u)$ is non-decreasing. \qed \bs

\noindent\underline {\bf(A3)}: \underline{The non free loop
property}: The switching costs $\underline{g}_{ik} $ and $
\bar{g}_{jl} $ are non-negative, jointly continuous in $(t,x)$,
belong to $\Pi_g$ and satisfy the following condition: \ms \no For
any loop in $A^1\times A^2$, i.e., any sequence of pairs
$(i_1,j_1),\ldots,(i_N,j_N)$
 of $A^1\times A^2$ such that $(i_N,j_N)=(i_1,j_1)$, \mbox{card}$\{
(i_1,j_1),\ldots,(i_N,j_N)\}=N-1$ and $\forall \,\,q=1,\ldots,N-1$,
either $i_{q+1}=i_q$ or $j_{q+1}=j_q$, we have:
\be\label{nonfreeloop3}\forall (t,x)\in \esp,\,\,
\sum_{q=1,N-1}\varphi_{i_qj_{q}}(t,x)\neq 0, \ee where for any
$q=1,\ldots,N-1,\,\,
\varphi_{i_qj_{q}}(t,x)=-\underline{g}_{i_qi_{q+1}}(t,x)\ind_{i_q\neq
i_{q+1}}+\bar{g}_{j_qj_{q+1}}(t,x)\ind_{j_q\neq j_{q+1}}$. \qed \bs

\noindent\underline {\bf(A4)}: The functions $h^{ij}:
\R^k\rightarrow \R$ are continuous w.r.t. $x$, belong to class
$\Pi_g$ and satisfy:
$$\forall \ijg, \,\,x\in \R^k,~~\max\limits_{k\in
A^1_i}(h^{kj}(x)-\underline{g}_{ik}(T,x))\leq h^{ij}(x)\leq
\min\limits_{l\in A^2_j}(h^{il}(x)+\overline{g}_{jl}(T,x)).\qed $$
To begin with let us point out that the non-local terms ${\cal
I}(t,x,\phi)$ and ${\cal I}_{ij}(t,x,\phi)$ introduced previously
are well defined under Assumptions (A0) since for any function
$\phi$ of class $\cC^{1,2}$, by the mean value theorem, we have
$$|\phi(t,x+\beta(x,e))-\phi(t,x)-
D_x\phi(t,x)\beta(e,x)|\leq C^{(1)}_{t,x}|\beta(x,e)|^2\leq
C^{(1)}_{t,x}(1\wedge|e|)^2,$$ and
$$|\g^{ij}(x,e)(\phi(t,x+\beta(x,e))-\phi(t,x))|\leq
C^{(2)}_{t,x}|\beta(x,e)\g^{ij}(x,e)|\leq
C^{(2)}_{t,x}(1\wedge|e|^2)$$where $C^{(1)}_{t,x}$ and
$C^{(2)}_{t,x}$ are bounded constants. They are the bounds of the
first and second derivatives of $y\mapsto \phi(t,y)$ in
$B(x,K_\beta)$ where $K_\beta$ is a bound of the function $\beta$.
\qed \bs

Let us consider now the following SDE of jump-diffusion type
($\tx$): \be\label{cmSDE}\ba{l} X^{t,x}_s=x+\int_t^s
b(r,X^{t,x}_r)dr+\int^s_t\sigma(r,X^{t,x}_r)dW_r+\int^s_t\int_E\beta(X^{t,x}_{r-},e)
\hat{N}(dr,de),~s\in [t,T].\ea \ee The existence and uniqueness of
the solution
$X^{t,x}:=(X^{t,x}_s)_{s\in [t,T]}$ follows from \cite{fujiwara}.\\
We now precise the definition of the viscosity solution of system
(\ref{SIPDE}). First, for a locally bounded function $u$: $(t,x)\in
[0,T]\times \R^k\mps u(t,x)\in \R$, we define its lower
semi-continuous (lsc for short) envelope $u_*$, and upper
semi-continuous (usc for short) envelope $u^*$ as
following:$$u_*(t,x)=\varliminf\limits_{(t',x')\rightarrow
(t,x),~t'<T}u(t',x'), ~~~
u^*(t,x)=\varlimsup\limits_{(t',x')\rightarrow
(t,x),~t'<T}u(t',x').$$
\begin{defi}A function $\vec{u}(t,x)=(u^{ij}(t,x))_{(i,j)\in
A^1\times A^2}:[0,T]\times \R^k\rightarrow \R^{m_1\times m_2}$ such
that $u^{ij}$ is lsc
(resp. usc) and belongs to $\Pi_g$, is said to be a viscosity supersolution  (resp.
subsolution) of (\ref{SIPDE}) if for any $\ij$:
\ms

- $u^{ij}(T,x_0)\geq \mb{(resp. }\leq ) \,\,h^{ij}(x_0), \forall x_0\in \R^k.$

- for any $(t_0,x_0)\in (0,T)\times \R^k$  and any test function
$\phi\in \cC^{1,2}([0,T]\times \R^k)$ such  that $(t_0,x_0)$ is a
global minimum (resp.  maximum) point of $u^{ij}-\phi$ and
$u^{ij}(t_0,x_0)=\phi(t_0,x_0)$ ,
$$\left\{
    \begin{array}{ll}
\min\{(u^{ij}-L^{ij}[\vec{u}])(t_0,x_0);\q
\max\{(u^{ij}-U^{ij}[\vec{u}])(t_0,x_0);\q\\\\
\qq -\partial_t\phi(t_0,x_0)-b(t_0,x_0)^\top
D_x\phi(t_0,x_0)-\frac{1}{2}Tr[\sigma
\sigma^\top(t_0,x_0)D^2_{xx}\phi(t_0,x_0)]-I(t_0,x_0,\phi)\\\\
\qq -g^{ij}(t_0,x_0,(u^{kl}(t_0,x_0))_{(k,l)\in A^1\times
A^2},\sigma^\top(t_0,x_0)D_x\phi(t_0,x_0),I_{ij}(t_0,x_0,\phi))\}\}
\geq 0 \,\,(resp. \leq 0). \end{array} \right.$$ A function
$\vec{u}=(u^{ij}(t,x))_{(i,j)\in A^1\times A^2}$ of $\Pi_g$ is
called a viscosity solution of (\ref{SIPDE}) if
$(u^{ij}_*(t,x))_{(i,j)\in A^1\times A^2}$ (resp.
$(u^*_{ij}(t,x))_{(i,j)\in A^1\times A^2}$) is a viscosity
supersolution  (resp.  subsolution) of (\ref{SIPDE}).
\end{defi}
\ms

\begin{rem} By taking $\bar g_{jl}\equiv +\infty$ (resp.
 $\underline g_{ik}\equiv +\infty$) for any $j,l\in A^2$ (resp.
 $i,k\in A^1$) we obtain the definition of a viscosity solution of the system of
 variational inequalities with interconnected lower (resp. upper)
 obstacles.\qed
\end{rem}
\section{Uniqueness of the viscosity solution
of the non-local SVI  (\ref{SIPDE})}In this section we will show the uniqueness of
the viscosity solution of (\ref{SIPDE}) as a corollary of a
comparison result. In the same way with \cite{boualem}, Lemma 4.1,
we can prove the following lemma.
\begin{lem}
Let $(u^{ij})_{(i,j)\in A^1\times A^2}$ (resp.
$({w}^{ij})_{(i,j)\in A^1\times A^2}$) be an $usc$
subsolution  (resp.  $lsc$ supersolution) of (\ref{SIPDE}) which
belongs to $\Pi_g$. For  $(t,x)\in [0,T]\times R^k$, let
$\Gamma(t,x)$ be the following set:$$\Gamma(t,x):=\{(i,j)\in
A^1\times A^2, u^{ij}(t,x)-w^{ij}(t,x)=\max\limits_{(k,l)\in
A^1\times A^2}(u^{kl}(t,x)-w^{kl}(t,x))\}.$$ Then there exists
$(i_0,j_0)\in\Gamma(t,x)$ such that \be\label{41} u^{i_0
j_0}(t,x)>L^{i_0 j_0}[\vec{u}](t,x)\mbox{ and }w^{i_0
j_0}(t,x)<U^{i_0 j_0}[\vec{w}](t,x).\ee
\end{lem}
We now give the comparison result.
\begin{thm}
Let $\vec u=(u^{ij})_{(i,j)\in A^1\times A^2}$ (resp. $\vec
w=({w}^{ij})_{(i,j)\in A^1\times A^2}$) be an $usc$ subsolution
(resp.  $lsc$ supersolution) of (\ref{SIPDE}) which belongs to
$\Pi_g$. Then it holds that:
$$\forall \,\,\ijg\, \mbox{ and }\tx,\,\,u^{ij}(t,x)\leq w^{ij}(t,x).$$
\end{thm}
\begin{proof}Let $C$ and $\rho$  be positive constants, which exist thanks to the polynomial
growth of $\vec u$ and
$\vec w$, such that for any $\ijg$,
$$
|u^{ij}(t,x)|+|w^{ij}(t,x)|\leq C(1+|x|^\rho).
$$ There exists a positive constant $\l_0$ such that for any
$\l\geq \l_0$ and $\theta >0$, $\tilde
u^{\theta}:=(u^{ij}(t,x)-\theta e^{\l t}(1+|x|^{2\rho+2})_{\ijg}$
(resp. $\tilde w^\theta=(w^{ij}(t,x)-\theta e^{\l
t}(1+|x|^{2\rho+2})_{\ijg}$) is a subsolution (resp. supersolution)
of (\ref{SIPDE}) (see \cite{zhao}, pp.1634). Therefore it is enough
to show that $\tilde u^{\theta}\leq \tilde w^\theta$ and to take the
limit as $\theta \rightarrow 0$ to obtain the desired result.
Finally with the previous statement w.l.o.g one can assume that
there exists a real constant $\bar R>0$ such that for any $|x|\geq
\bar R$, $w^{ij}(t,x)>0$ (resp. $u^{ij}(t,x)<0$) for any $(i,j)$ and
$t\in [0,T]$. \ms

The proof now will be divided into two steps.

\ms \no {\bf Step 1}: Let $C_{ij}$ be the Lipschitz constant of
$g^{ij}$ w.r.t. $\vec{y}$. We first assume that there exists a
constant $\lambda_1>m_1m_2. \sum_{\ijg}C_{ij}$ such that for any
$\ijg$, \be\label{49}g^{ij}(t,x,[\vec{y
}^{ij},\zeta_1],z,q)-g^{ij}(t,x, [\vec{y }^{ij},\zeta_2],z,q)\leq
-\lambda(\zeta_1- \zeta_2),\ee for any $\zeta_1\geq \zeta_2$ in $\R$ and
$(t,x,\vec y,z,q)$ in their respective spaces. \ms

\noindent We proceed by contradiction. Assume there exists
$(\bar t,\bar x)\in [0,T]\times \R^k$ such that:
\be\label{46}\max\limits_{\ijg}(u^{ij}-w^{ij})(\bar t,\bar x)>0.\ee
Therefore there exists $(t^*,x^*)\in [0,T[\times B(0,\bar R)$
($B(0,\bar R)$ is the open ball in $\R^k$ centered in $0$ and of
radius $\bar R$ and w.l.o.g. we assume $t^*>0$) such
that:\be\label{48} \max\limits_{(t,x)\in[0,T]\times
\R^k}\max\limits_{\ijg}\{(u^{ij}-w^{ij})(t,x)\}=\max\limits_{(t,x)\in[0,T]\times
B(0,\bar
R)}\max\limits_{\ijg}\{(u^{ij}-w^{ij})(t,x)\}=\max\limits_{\ijg}\{(u^{ij}-w^{ij})
(t^*,x^*)\}>0. \ee Next let $(i_0,j_0)\in \Gamma(t^*,x^*)$ that
satisfies (\ref{41}). For $\eps>0$, $\rho>0$, let
$\Phi^{i_0,j_0}_{\eps,\rho}$ be the function defined as
follows:$$\Phi^{i_0j_0}_{\eps,\rho}(t,x,y):=(u^{i_0j_0}(t,x)-w^{i_0j_0}
(t,y))-\frac{|x-y|^2}{\eps}-|t-t^*|^2-\rho|x-x^*|^4.$$ Let
$(t_0,x_0,y_0)$ be such that
$$\Phi^{i_0j_0}_{\eps,\rho}(t_0,x_0,y_0)=\max\limits_{(t,x,y)\in [0,T]\times
\bar{B}(0,\bar R)^2}
\Phi^{i_0j_0}_{\eps,\rho}(t,x,y)=\max\limits_{(t,x,y)\in [0,T]\times
\R^{k+k}}\Phi^{i_0j_0}_{\eps,\rho}(t,x,y).$$ The second equality is
valid since when $|x|\geq \bar R$ (resp. $|y|\geq \bar R$),
$u^{i_0j_0}(t,x)<0$ (resp. $w^{i_0j_0}(t,y)>0$). On the other hand
$(t_0,x_0,y_0)$ depends actually on $\eps$ and $\rho$ which we omit
for sake of simplicity. Next as usual we have \be\label{limitx}
\lim\limits_{\eps\rightarrow
0}(t_0,x_0,y_0)=(t^*,x^*,x^*),~~\lim\limits_{\eps\rightarrow
0}\frac{|x_0-y_0|^2}{\eps}=0, \ee and \be \label{limitx2}
\lim\limits_{\eps \rw 0}
(u^{i_0j_0}(t_0,x_0),w^{i_0j_0}(t_0,y_0))=(u^{i_0j_0}({t}^*,{x}^*),w^{i_0j_0}({t}^*,{x}^*)).\ee
Therefore for $\eps$ small enough it holds
\be\label{411}u^{i_0j_0}(t_0,x_0)>\max\limits_{k\in
A^1_{i_0}}(u^{kj_0}(t_0,x_0)-\underline{g}_{i_0 k}(t_0,x_0)),\ee and
\be\label{412}w^{i_0j_0}(t_0,y_0)<\min\limits_{l\in
A^2_{j_0}}(w^{i_0l}(t_0,y_0)+\overline{g}_{j_0 l}(t_0,y_0)).\ee Once
more for $\eps$ small enough, we are able to apply Jensen-Ishii's
Lemma for non-local operators established by Barles and Imbert
(\cite{guy}, pp. 583) (one can see also \cite{Bis}, Lemma 4.1, pp.
64) with  $u^{i_0j_0}$, $w^{i_0j_0}$ and
$\phi(t,x,y):=\frac{|x-y|^2}{\eps}+|t-t^*|^2+\rho|x-x^*|^4$ at
$(t_0,x_0,y_0)$. For any $\d\in (0,1)$ there exist
$p^\eps_u,q^\eps_u$ in $\R$, $p^\eps_w$, $q^\eps_w$ in $\R^k$ and
$M^\eps_u$, $M^\eps_w$ two symmetric non-negative matrices of
$\R^{k\times k}$ such that: \ms

\no (i) \be\begin{aligned}
p^\eps_u-p^\eps_w=\partial_t{\phi}(t_0,x_0,y_0),
~~q^\eps_u=\partial_x{\phi}(t_0,x_0,y_0),~~q^\eps_w=-\partial_y{\phi}(t_0,x_0,y_0)\end{aligned}\ee
and
 \be \label{unique7} \left(\begin{matrix}
        M^0_u & 0 \cr
      0 & -M^0_w
      \end{matrix}\right)\leq \left(\begin{matrix}
     D^2_{xx}\psi_\rho(t_0,x_0)& 0 \cr
      0 & 0
      \end{matrix}\right)+ \frac{4}{\eps}\left(\begin{matrix}
      I_k & -I_k\cr
      -I_k & I_k
      \end{matrix}\right)\,\ee \mbox{where }$\psi_\rho(t,x):=\rho|x-x^*|^4+|t-t^*|^2$\,\,;\\
\be\label{uniqueinequa1}
\begin{aligned}
\mbox{
(ii)}\,\,&-p^\eps_u-b(t_0,x_0)^\top q^\eps_u-\frac{1}{2}Tr(\sigma(t_0,x_0)^\top
M^\eps_u\sigma(t_0,x_0))-g^{i_0j_0}(t_0,x_0,(u^{ij}(t_0,x_0))_{(i,j)\in A^1\times A^2},
\sigma(t_0,x_0)\tp q^\eps_u,\\
&I^{1,\d}_{i_0j_0}(t_0,x_0,\phi(t_0,.,y_0))+I^{2,\d}_{i_0j_0}(t_0,x_0,u^{i_0j_0}))
-I^{1}_\d(t_0,x_0,\phi(t_0,.,y_0))-I^{2}_\d(t_0,x_0,q^\eps_u,u^{i_0j_0})\leq
0\,;
\end{aligned}
\ee \be\label{uniqueinequa2}
\begin{aligned}
\mbox{
(iii)}\,\,&-p^\eps_w-b(t_0,y_0)\tp q^\eps_w-\frac{1}{2}Tr(\sigma(t_0,y_0)^\top
M^\eps_w\sigma(t_0,y_0))-g^{i_0j_0}(t_0,y_0,(w^{ij}(t_0,y_0))_{(i,j)\in A^1\times A^2},
\sigma(t_0,y_0)\tp q^\eps_w,\\
&I^{1,\d}_{i_0j_0}(t_0,y_0,-\phi(t_0,
x_0,.))+I^{2,\d}_{i_0j_0}(t_0,x_0,w^{i_0j_0}))
-I^{1}_\d(t_0,y_0,-\phi(t_0,
x_0,.))-I^{2}_\d(t_0,y_0,q^\eps_w,w^{i_0j_0})\geq 0.
\end{aligned}\ee
Next we are going to provide estimates for the non-local terms. By the same
argument as in \cite{zhao} pp.1645, we have:
\be\label{estimlocal2}\begin{array}{ll}
I^{2}_\d(t_0,x_0,q^\eps_u,u^{i_0j_0}) -I^{2}_\d(t_0,y_0,
q^\eps_w,w^{i_0j_0})&\leq C\frac{|x_0-y_0|^2}{\eps}\int_{ |e|\geq
\d}(1\wedge|e|)^2n(de)+I^2_\delta(t_0,x_0,D_x\psi_\rho(t_0,x_0),\psi_\rho)\\
{}&\leq
{C}\frac{|x_0-y_0|^2}{\eps}+I^2_\delta(t_0,x_0,D_x\psi_\rho(t_0,x_0),\psi_\rho).\end{array}\ee
On the other hand
$$
D^2_{xx}\phi(t,x,y)={2}{\eps^{-1}}I_k+D^2_{xx}\psi_\rho(t,x,y),\,\,
D^2_{yy}\phi(t,x,y)={2}{\eps^{-1}}I_k
$$ and by Taylor's expansion we have
$$\begin{array}{c}
\phi(t,x+\beta(x,e),y)-\phi(t,x)-D_x\phi(t,x,y)\beta(x,e)=\int_0^1(1-t)\b(x,e)^\top
D^2_{xx}\phi(t,x+t\beta(x,e),y)\beta(x,e)dt.  \end{array}
$$
It implies that
$$\begin{array}{c}
I^{1}_\d(t_0,x_0,\phi(t_0,.,y_0))=\int_{|e|\leq \d} n(de)\int_0^1(1-t)\b(x_0,e)^\top
D^2_{xx}\phi(t_0,x_0+t\beta(x_0,e),y_0)\beta(x_0,e)dt
\end{array}
$$and similarly
$$\begin{array}{c}
I^{1}_\d(t_0,y_0,-\phi(t_0,x_0,.))=-\int_{|e|\leq \d} n(de)\int_0^1(1-t)\b(y_0,e)^\top
D^2_{yy}\phi(t_0,x_0, y_0+t\beta(y_0,e))\beta(y_0,e)dt.
\end{array}
$$
Consequently it holds that \be\label{estimlocal1}\begin{array}{l}
\lim_{\delta\rightarrow
0}I^{1}_\d(t_0,x_0,\phi(t_0,.,y_0))=\lim_{\delta\rightarrow
0}I^{1}_\d(t_0,y_0,- \phi(t_0,x_0,.))=0.\end{array} \ee Next by the
definition of $(t_0,x_0,y_0)$, we have
$$\begin{array}{l}u^{i_0j_0}(t_0,x_0+
\beta(x_0,e)) -u^{i_0j_0}(t_0,x_0)\\\qq\qq \leq w^{i_0,j_0}(t_0,y_0+
\beta(y_0,e))- w^{i_0j_0}(t_0,y_0)+ \psi_\rho(t_0,x_0+\beta(x_0,e))-\psi_\rho(t_0,x_0)
\\
\qq\qq\qq  +\eps^{-1}\{|\beta(x_0,e) -\beta(y_0,e)|^2-2(x_0-y_0)(\beta(x_0,e)-\beta(y_0,e))\}.
\end{array}$$
Since $\gamma^{i_0j_0}$ is nonnegative, and by the assumptions on
$\beta$ (see (A0)-(i)), for any $\d>0$,
\be\label{7141}I^{2,\d}_{i_0j_0}(t_0,x_0,u^{i_0j_0})-I^{2,\d}_{i_0j_0}(t_0,y_0,w^{i_0j_0})\leq
I^{2,\d}_{i_0j_0}(t_0,x_0,\psi_\rho)+O({\eps^{-1}}{|x_0-y_0|^2}),\ee
and it is easy to check that
\be\label{7141(ii)}\begin{array}{c}|I^{2,\d}_{i_0,j_0}(t_0,x_0,\psi_\rho)|\leq C\rho
\int_{|z|\geq \d}1\wedge |e|^2n(de).\end{array}\ee
 On the other hand, since $\phi$ is a ${\cal C}^2$-function then, using once more Taylor's expansion to obtain:
\be \label{equnicite}\begin{array}{ll}
&I^{1,\d}_{i_0j_0}(t_0,x_0,\phi(t_0,.,y_0))= \int_{|e|\leq
\d}\{\phi(t_0,x_0+ \beta(x_0,e),y_0)-
\phi(t_0,x_0,y_0)\}\gamma^{i_0j_0}(x_0,e)n(de)
\\\\
{}&= \int_{|e|\leq \d}n(de)\int_0^1dt(1-t)\beta(x_0,e)^\top[ \frac{2}{\eps}\{x_0+t\beta(x_0,e)-y_0\}+D_x\psi_\rho (t_0,x_0+t\beta(x_0,e))]
\gamma^{i_0j_0}(x_0,e)\\\\{}&
=\frac{x_0-y_0}{\eps}\int_{|e|\leq \d}\beta(x_0,e)\tp \gamma^{i_0j_0}(x_0,e)n(de)+\,\,\frac{1}{3}
\int_{|e|\leq \d}|\beta(x_0,e)|^2\gamma^{i_0j_0}(x_0,e)n(de) \\\\{}&\qq\qq\qq+\int_{|e|\leq \d}n(de)\beta_(x_0,e)\tp D_x\psi_\rho(t_0,x_0+t\beta(x_0,e))\gamma^{i_0j_0}(x_0,e)
\end{array}\end{equation}
and \be \label{equnicite2}\begin{array}{ll}
I^{1,\d}_{i_0j_0}(t_0,y_0,- \phi(t_0,x_0,.))&=- \int_{|e|\leq
\d}\{\phi(t_0,x_0, y_0+ \beta(y_0,e))-
\phi(t_0,x_0,y_0)\}\gamma^{i_0j_0}(y_0,e)n(de)
\\\\{}&=\frac{x_0-y_0}{\eps}\int_{|e|\leq \d}\beta(y_0,e)\tp \gamma^{i_0j_0}(y_0,e)n(de)-\frac{1}{3}
\int_{|e|\leq \d}|\beta(x_0,e)|^2\gamma^{i_0j_0}(x_0,e)n(de)
\end{array}\ee which implies that
\be \label{equnicite3}\begin{array}{l}\lim_{\d \rightarrow 0}
I^{1,\d}_{i_0j_0}(t_0,x_0,\phi(t_0,.,y_0))= \lim_{\d \rightarrow 0}
I^{1,\d}_{i_0j_0}(t_0,y_0,- \phi(t_0,x_0,.))=0.\end{array}\ee Making
now the difference between (\ref{uniqueinequa1}) and
(\ref{uniqueinequa2}) yields \be \label{eqdifference}
\begin{array}{ll}
&-(p^\eps_u-p^\eps_w)-[b(t_0,x_0)\tp q^\eps_u-b(t_0,y_0)\tp
q^\eps_w]-\frac{1}{2}\{Tr[\sigma(t_0, x_0)\tp
M^\eps_u\sigma(t_0, x_0)-\sigma(t_0,y_0)\tp M^\eps_w\sigma(t_0, y_0)]\}\\
&-[g^{i_0j_0}(t_0,x_0,(u^{ij}(t_0,x_0))_{(i,j)\in A^1\times A^2},\sigma(t_0,x_0)\tp
q^\eps_u,I^{1,\d}_{i_0j_0}(t_0,x_0,\phi(t_0,.,y_0))+I^{2,\d}_{i_0j_0}(t_0,x_0,u^{i_0j_0}))\\
&-g^{i_0j_0}(t_0,y_0,(w^{ij}(t_0,y_0))_{(i,j)\in A^1\times
A^2},\sigma(t_0,y_0)\tp
q^\eps_w,I^{1,\d}_{i_0j_0}(t_0,y_0,-\phi(t_0,
x_0,.))+I^{2,\d}_{i_0j_0}(t_0,y_0,w^{i_0j_0}))]\\
&-I^{1,\delta}(t_0,x_0,\phi(t_0,.,y_0))+I^{1,\delta}(t_0,y_0,-
\phi(t_0,x_0,.)) -I^{2,\delta}(t_0,x_0,q^\eps_u,u^{i_0j_0})
+I^{2,\delta}(t_0,y_0, q^\eps_w,w^{i_0j_0})\leq 0.
\end{array}\ee
As $q\in \R\mapsto g^{ij}(t,x,\vec y,z,q)$ is non-decreasing and
Lipschitz then, by linearization procedure of Lipschitz functions,
there exists a bounded non-negative quantity $\Xi$ (which depends on
$\eps$, $\d$, etc.) such that
 \begin{align*}
&-[g^{i_0j_0}(t_0,x_0,(u^{ij}(t_0,x_0))_{(i,j)\in A^1\times A^2},
\sigma(t_0,x_0)\tp q^\eps_u,I^{1,\d}_{i_0j_0}(t_0,x_0,
\phi(t_0,.,y_0))+I^{2,\d}_{i_0j_0}(t_0,x_0,u^{i_0j_0}))\\
&\qq -g^{i_0j_0}(t_0,y_0,(w^{ij}(t_0,y_0))_{(i,j)\in A^1\times
A^2},\sigma(t_0,y_0)\tp
q^\eps_w),I^{1,\d}_{i_0j_0}(t_0,y_0,-\phi(t_0,
x_0,.))+I^{2,\d}_{i_0j_0}(t_0,y_0,w^{i_0j_0}))]\\
&=-[g^{i_0j_0}(t_0,x_0,(u^{ij}(t_0,x_0))_{(i,j)\in A^1\times A^2},
\sigma(t_0,x_0)\tp
q^\eps_u,I^{1,\d}_{i_0j_0}(t_0,x_0,\phi(t_0,.,y_0))+
I^{2,\d}_{i_0j_0}(t_0,x_0,u^{i_0j_0}))\\
&\qq-g^{i_0j_0}(t_0,y_0,(w^{ij}(t_0,y_0))_{(i,j)\in A^1\times
A^2},\sigma(t_0,y_0)\tp
q^\eps_w,I^{1,\d}_{i_0j_0}(t_0,x_0,\phi(t_0,.,y_0))+I^{2,\d}_{i_0j_0}(t_0,x_0,u^{i_0j_0}))]
        \\&
\qq -\Xi \times [I^{1,\d}_{i_0j_0}(t_0,x_0,
\phi(t_0,.,y_0))+I^{2,\d}_{i_0j_0}(t_0,x_0,u^{i_0j_0})-I^{1,\d}_{i_0j_0}(t_0,y_0,
-\phi(t_0, x_0,.))-I^{2,\d}_{i_0j_0}(t_0,y_0,w^{i_0j_0})]
  \end{align*}
On the other hand, once more by a linearization procedure and taking
into account of (\ref{49}), the first term of the right-hand side of
the previous equality verifies:
 \begin{align*}
&[g^{i_0j_0}(t_0,x_0,(u^{ij}(t_0,x_0))_{(i,j)\in A^1\times A^2},
\sigma(t_0,x_0)\tp
q^\eps_u,I^{1,\d}_{i_0j_0}(t_0,x_0,\phi(t_0,.,y_0))+
I^{2,\d}_{i_0j_0}(t_0,x_0,u^{i_0j_0}))\\
&\qq -g^{i_0j_0}(t_0,y_0,(w^{ij}(t_0,y_0))_{(i,j)\in A^1\times
A^2},\sigma(t_0,y_0)\tp
q^\eps_w,I^{1,\d}_{i_0j_0}(t_0,x_0,\phi(t_0,.,y_0))+I^{2,\d}_{i_0j_0}(t_0,x_0,u^{i_0j_0}))]\\&
\le  -\l (u^{i_0j_0}(t_0,x_0)- w^{i_0j_0}(t_0,y_0))+\sum_{(i,j)\neq
(i_0,j_0)}\xi^{ij}_{t_0,x_0,y_0,\d}(u^{ij}(t_0,x_0)-w^{ij}(t_0,y_0))\\&\qq
 +\eta^{i_0j_0}_{t_0,x_0,y_0,\d}(\sigma(t_0,x_0)\tp
q^\eps-\sigma(t_0,y_0)\tp q^\eps_w)+\sup_{\vec
y,z,q}|g^{i_0j_0}(t_0,x_0,\vec y,z,q)-g^{i_0j_0}(t_0,y_0,\vec
y,z,q)|
  \end{align*}
 where   $\xi^{ij}_{t_0,x_0,y_0,\d}$ is bounded non-negative quantity
 (positiveness stems from (A2)) by $C_{i_0j_0}$, and
$\eta^{ij}_{t_0,x_0,y_0,\d}$ is a bounded quantity by the Lipschitz constant
 of $g^{i_0j_0}$ w.r.t. $z$. Therefore from (\ref{eqdifference}) we
 deduce:
\be \label{eqdifferencex}
\begin{array}{ll}
&\l (u^{i_0j_0}(t_0,x_0)- w^{i_0j_0}(t_0,y_0)) \leq \\&\qq
(p^\eps_u-p^\eps_w)+[b(t_0,x_0)\tp q^\eps_u-b(t_0,y_0)\tp
q^\eps_w]+\frac{1}{2}\{Tr[\sigma(t_0, x_0)\tp
M^\eps_u\sigma(t_0, x_0)-\sigma(t_0,y_0)\tp M^\eps_w\sigma(t_0, y_0)]\}\\
& \qq+\q\Xi .[I^{1,\d}_{i_0j_0}(t_0,x_0,
\phi(t_0,.,y_0))-I^{1,\d}_{i_0j_0}(t_0,y_0, -\phi(t_0,
x_0,.))+I^{2,\d}_{i_0j_0}(t_0,x_0,u^{i_0j_0})-I^{2,\d}_{i_0j_0}(t_0,y_0,w^{i_0j_0})]\\&
\qq +\q I^{1,\delta}(t_0,x_0,\phi(t_0,.,y_0))-I^{1,\delta}(t_0,y_0,-
\phi(t_0,x_0,.)) +I^{2,\delta}(t_0,x_0,q^\eps_u,u^{i_0j_0})
-I^{2,\delta}(t_0,y_0, q^\eps_w,w^{i_0j_0})\\&\qq+\q \sum_{(i,j)\neq
(i_0,j_0)}\xi^{ij}_{t_0,x_0,y_0}(u^{ij}(t_0,x_0)-w^{ij}(t_0,y_0))+
\eta^{i_0j_0}_{t_0,x_0,y_0}(\sigma(t_0,x_0)\tp
q^\eps-\sigma(t_0,y_0)\tp q^\eps_w)\\&\qq+\q \sup_{\vec
y,z,q}|g^{i_0j_0}(t_0,x_0,\vec y,z,q)-g^{i_0j_0}(t_0,y_0,\vec
y,z,q)|.
\end{array}\ee
Next recall (\ref{estimlocal2}),
(\ref{estimlocal1}),(\ref{7141})-(ii) and (\ref{equnicite3}), then
take the limit superior as $\d \rightarrow 0$ then the limit
superior as $\rho\rightarrow 0$ to obtain: \be \label{eqdifferencey}
\begin{array}{ll}
&\l (u^{i_0j_0}(t_0,x_0)- w^{i_0j_0}(t_0,y_0)) \leq \\&\qq
(p^\eps_u-p^\eps_w)+[b(t_0,x_0)\tp q^\eps_u-b(t_0,y_0)\tp
q^\eps_w]+\frac{1}{2}\{Tr[\sigma(t_0, x_0)\tp M^\eps_u\sigma(t_0,
x_0)-\sigma(t_0,y_0)\tp M^\eps_w\sigma(t_0, y_0)]\} \\&\qq+\q
\sum_{(i,j)\neq (i_0,j_0)}C_{ij}(u^{ij}(t_0,x_0)-w^{ij}(t_0,y_0))^++
C_{i_0j_0}(\sigma(t_0,x_0)\tp q^\eps-\sigma(t_0,y_0)\tp
q^\eps_w)\\&\qq+\q \sup_{\vec y,z,q}|g^{i_0j_0}(t_0,x_0,\vec
y,z,q)-g^{i_0j_0}(t_0,y_0,\vec
y,z,q)|+|O({\eps^{-1}}{|x_0-y_0|^2})|+C{\eps}^{-1}{|x_0-y_0|^2} .
\end{array}\ee
Finally by the continuity of $g^{i_0j_0}$ (see (A1)-(i)),  using the
properties satisfied by
$p^\eps_u,p^\eps_w,q^\eps_u,q^\eps_w,M^\eps_u$ and $M^\eps_w$,
sending $\eps$ to $0$ and taking into account of
(\ref{limitx})-(\ref{limitx2}) to obtain (in a classical way) that:
\begin{align*}
\lambda(u^{i_0j_0}(t^*,x^*)-w^{i_0j_0}(t^*,x^*))\leq
\sum\limits_{(i,j)\neq(i_0,j_0)}C_{ij}(u^{ij}(t^*,x^*)-w^{ij}(t^*,x^*))^+
\end{align*} which is contradictory by the definitions of $\l$ and $(t^*,x^*)$. Thus for any $\ij$,
$u^{ij}\leq w^{ij}$. Note that we have used the fact that $u^{ij}$
(resp $ w^{ij}$) is usc (resp. lsc) when we take the limit as $\eps \rightarrow 0$ to deduce the last inequality.\ms

\no {\bf Step 2}: The general case. If $(u^{ij})_{\ij} $(resp. $(w^{ij})_{\ij}$) is a subsolution (resp. supersolution) of (\ref{SIPDE}) in the class $\Pi_g$, then for any $\l \in \R$, $(e^{-\l t}u^{ij})_{\ij}$ (resp. $(e^{-\l t}w^{ij})_{\ij}$) is a subsolution (resp. supersolution) of a system of type (\ref{SIPDE}) but associated with \\
 $\{(e^{-\l t}g^{ij}(t,x,(e^{\l t}u^{kl})_{(k,l)\in A^1\times A^2},
 e^{\l t}\sigma(t,x)^\top
 D_xu^{ij}(t,x), I_{ij}(t,x,e^{\l t}u^{ij}))-\l u^{ij}(t,x)$ \\
 $ (e^{-\l t}\underline g_{ik}(t,x))_{i,k\in A^1}, (e^{-\l t}\bar g_{jl}(t,x))_{j,l\in A^2}, (e^{-\l T}h^{ij}(x))_{\ij})\}.$ But in choosing appropriately the constant $\l$ we get that the functions $\tilde g^{ij}(t,x,\vec y,z^{ij},q):=
 e^{-\l t}g^{ij}(t,x,e^{\l t}\vec y, e^{\l t}z^{ij},q)-\l y^{ij}$ verify (\ref{49}). Therefore by the result of Step 1 we have for any $\ij$,
$e^{-\l t}u^{ij}\leq e^{-\l t}w^{ij}$, whence the desired result.
\end{proof}
As a by-product we have:
\begin{cor}\label{coruni} The system of variational inequalities with bilateral obstacles
(\ref{SIPDE}) has at most one viscosity solution in the class
$\Pi_g$ which is moreover necessarily continuous. \ms
\end{cor}
Finally a remark in the case when we have only lower or upper
interconnected obstacles.
\begin{rem} \label{unicitecasunebarriere} If we assume that $\bar g_{jk}\equiv +\infty$ (resp.
$\underbar g_{il}\equiv +\infty$) for any $j,k\in A^2$ (res. $i,l\in
A^1$), then system (\ref{SIPDE}) turns into  of type with one lower
(resp. upper) interconnected obstacles. As the non free loop
property is satisfied by $(\underbar g_{il})_{i,l\in A^1}$ (resp.
$(\bar g_{jk})_{j,k\in A^2}$) then the system of variational
inequalities with lower (resp. upper) interconnected obstacles has
at most one viscosity solution in the class $\Pi_g$ which is
moreover necessarily continuous. \qed
 \end{rem}
\section{Approximating schemes and BSDEs} For $n,m\geq 0$, let
$(Y^{ij,n,m},Z^{ij,n,m},U^{ij,n,m})_{(i,j)\in A^1\times A^2}$ be
the solution of the following system of BSDEs:
\be
\label{approximation}\left\{
    \begin{array}{ll}
      (Y^{ij,n,m},Z^{ij,n,m},U^{ij,n,m})\in
      \cS^2\times\cH^2\times\cH^2(\hat{N});\\
      dY^{ij,n,m}_s=-f^{ij,n,m}(s,X^{t,x}_s,(Y^{kl,n,m}_s)_{(k,l)\in A^1\times
      A^2},Z^{ij,n,m}_s,U^{ij,n,m}_s)ds+\\ \qq\qq\qq\qq
      Z^{ij,n,m}_sdB_s+\int_E U^{ij,n,m}_s(e)\hat{N}(ds,de),~s\leq T;\\
      Y^{ij,n,m}_T=h^{ij}(X^{t,x}_T),
    \end{array}
    \right.\ee
where \begin{align*}&f^{ij,n,m}(s,X^{t,x}_s,({y}^{ij})_{(i,j)\in
A^1\times A^2},{z},{u}):=g^{ij,n,m}(s,X^{t,x}_s,({y}^{kl})_{(k,l)\in
A^1\times A^2},{z},\pin
{u}(e)\g^{ij}(X^{t,x}_s,e)n(de))\\
&\qq \qq =g^{ij}(s,X^{t,x}_s,({y}^{kl})_{(k,l)\in A^1\times
A^2},{z},\pin
{u}(e)\g^{ij}(X^{t,x}_s,e)n(de))\\
&\qq \qq \qq \qq \qq \qq +n({y}^{ij}-\max\limits_{k\in
A^1_i}\{{y}^{kj}-\underline{g}_{ik}(s,X^{t,x}_s)\})^--m({y}^{ij}-\min\limits_{l\in
A^2_j}\{{y}^{il}+\overline{g}_{jl}(s,X^{t,x}_s)\})^+.\end{align*}
Let us notice that under Assumption {\bf (A1)}, the solution
$(Y^{ij,n,m},Z^{ij,n,m},U^{ij,n,m})_{(i,j)\in A^1\times A^2}$  of
(\ref{approximation}) exists and is unique (see e.g. \cite{tang2}
or \cite{guy}). Next, let us focus on the properties of the matrix
of processes $(Y^{ij,n,m})_{\ijg}$.
\begin{propo}
For any $(i,j)\in A^1\times A^2$ and $n,m\geq 0$ we have:

\noindent
(i) \be\label{monotonicity} P-a.s.,~~Y^{ij,n,m}\leq Y^{ij,n+1,m}~
and ~~Y^{ij,n,m+1}\leq Y^{ij,n,m}. \ee \noindent (ii) There exists a
deterministic continuous function $v^{ij,n,m}\in\Pi_g$ s.t., for
any $\tx$,\be\label{Feman-Kac}Y^{ij,n,m}_s=v^{ij,n,m}(s,X^{t,x}_s),~~s\in[t,T].\ee
\noindent (iii) for any $(t,x)\in [0,T]\times R^k$,
\be\label{v-monotonicity}v^{ij,n,m}(t,x)\leq v^{ij,n+1,m}(t,x)~
and ~~v^{ij,n,m+1}(t,x)\leq v^{ij,n,m}(t,x).~~~\ee
\end{propo}
\begin{proof} (i) Let $n$ and $m$ be fixed. We are going to use a result by X.Zhu (\cite{xuehong},
Theorem 3.1) related to comparison of solutions of multi-dimensional
BSDEs with jumps. It is enough to show that for any $t\in [0,T],
(y_{ij})_{\ijg}$, $(\overline{y}_{ij})_{\ijg}\in \R^{m_1\times
m_2}$, $(z_{ij})_{\ijg}$, $ (\overline{z}_{ij})_{\ijg}\in
(\R^d)^{m_1\times m_2}$ and $(u_{ij})_{\ijg},
(\overline{u}_{ij})_{\ijg}\in
(\cL^2_\R(E,\mathcal{B}_E,n)^{m_1\times m_2}$, there exists a
constant $C$ such
that:\begin{align*}&-4\sum\limits_{\ijg}y_{ij}^-\{f^{ij,n+1,m}(s,
X^{t,x}_s,
(y_{kl}^++\overline{y}_{kl})_{\klg},z_{ij},u_{ij})-f^{ij,n,m}(s,
X^{t,x}_s, (\overline{y}_{kl})_{\klg},\overline{z}_{ij},\overline{u}_{ij})\}\\
& \leq
2\sum\limits_{\ijg}\mathbbm{1}_{\{y^{ij}<0\}}|z_{ij}-\overline{z}_{ij}|^2+C\sum\limits_{(i,j)\in
A^1\times A^2}(y^-_{ij})^2\\
&\qq \qq+2\sum\limits_{\ijg}\int_E\mathbbm{1}_{\{y^{ij}\geq 0\}}|(y_{ij}+u_{ij}(e)-\overline{u}_{ij}(e))^-|^2n(de)\\
&\qq
\qq+2\sum\limits_{\ijg}\int_E\mathbbm{1}_{\{y^{ij}<0\}}[|(y_{ij}+u_{ij}(e)-
\overline{u}_{ij}(e))^-|^2-|y_{ij}^-|^2-2y_{ij}(u_{ij}(e)-\overline{u}_{ij}(e))]n(de)
\end{align*}
where for any $x\in \R$, $x^+=x\vee 0$ and $x^-=(-x)\vee 0$. But the
above inequality follows from the facts that for any $\ijg$:\\
a) $f^{ij,n,m}\leq f^{ij,n+1,m}$ ;\\
b) For any $(\theta_{kl})_{\klg}\in \R^{m_1\times m_2}$ such that
$\theta_{kl}\geq 0$ if $(k,l)\in \G^{-(i,j)})$ and $\theta_{ij}=0$,
$$ f^{ij,n,m}(s,X^{t,x}_s,(y_{kl}+\theta_{kl})_{\klg},
z_{ij},u_{ij})\geq
f^{ij,n,m}(s,X^{t,x}_s,(y_{kl})_{\klg},z_{ij},u_{ij}).$$ c) $f^{ij}$
depends only on $z_{ij}, u_{ij}$ and not on the other components
$z_{kl}, u_{kl},(k,l)\neq (i,j)$.\\
d)\begin{align*}&-4y_{ij}^-(f^{ij,n,m}(s, X^{t,x}_s,
(\overline{y}_{kl})_{\klg},\overline{z}_{ij},u_{ij})-f^{ij,n,m}(s,
X^{t,x}_s, (\overline{y}_{kl})_{\klg},\overline{z}_{ij},\overline{u}_{ij}))\\
&\qq \leq C (y^-_{ij})^2+2\int_E\mathbbm{1}_{\{y_{ij}\geq 0\}}|(y_{ij}+u_{ij}(e)-\overline{u}_{ij}(e))^-|^2n(de)\\
&\qq \qq
+2\int_E\mathbbm{1}_{\{y_{ij}<0\}}[|(y_{ij}+u_{ij}(e)-\overline{u}_{ij}(e))^-|^2-|y_{ij}^-
|^2-2y_{ij}(u_{ij}(e)-\overline{u}_{ij}(e))]n(de)
\end{align*}
Indeed a), b) and c) are easy to check. We just need to prove d). In
the case when $y_{ij}\geq 0$, this inequality is obvious. Next let
us focus on the case when $y_{ij}<0$. First note that by a
linearization procedure we have:
\begin{align*}&-4y_{ij}^-(f^{ij,n,m}(s, X^{t,x}_s,
(\overline{y}_{kl})_{\klg},\overline{z}_{ij},u_{ij})-f^{ij,n,m}(s,
X^{t,x}_s,
(\overline{y}_{kl})_{\klg},\overline{z}_{ij},\overline{u}_{ij}))\\
&=-4y_{ij}^-(g^{ij}(s, X^{t,x}_s,
(\overline{y}_{kl})_{\klg},\overline{z}_{ij},\int_E\gamma^{ij}(X^{t,x}_s,e){u}_{ij}(e)n(de)
)\\&\qq\qq\qq-g^{ij}(s, X^{t,x}_s,
(\overline{y}_{kl})_{\klg},\overline{z}_{ij},\int_E\gamma^{ij}(X^{t,x}_s,e)\overline{u}_{ij}(e)n(de)
)) \\&=4y_{ij}^-\times  \Xi_1\times
\int_E\gamma^{ij}(X^{t,x}_s,e)(\overline{u}_{ij}(e)-{u}_{ij}(e))n(de)
\end{align*}
where $\Xi_1$ is a non-negative quantity (since $g^{ij}$ is
nondecreasing in $q$) and bounded by the Lipschitz constant of
$g^{ij}$ w.r.t. $q$ and which depends on the other variables. Now
\begin{align*}
 &4y_{ij}^-\times \Xi_1\times
\pin\gamma^{ij}(X^{t,x}_s,e)(\overline{u}_{ij}(e)-{u}_{ij}(e))n(de)
\\&
=4y_{ij}^-\times \Xi_1\times \int_{{u}_{ij}-\overline{u}_{ij}<
-y_{ij}}\gamma^{ij}(X^{t,x}_s,e)(\overline{u}_{ij}(e)-{u}_{ij}(e))n(de)
\\&
\qq\qq +4y_{ij}^-\times \Xi_1\times
\int_{{u}_{ij}-\overline{u}_{ij}\ge
-y_{ij}}\gamma^{ij}(X^{t,x}_s,e)(\overline{u}_{ij}(e)-{u}_{ij}(e))n(de)
\\&\leq 4y_{ij}^-\times \Xi_1\times
\int_{{u}_{ij}-\overline{u}_{ij}<
-y_{ij}}\gamma^{ij}(X^{t,x}_s,e)(\overline{u}_{ij}(e)-{u}_{ij}(e))n(de).
\end{align*}
But for an appropriate constant $C$,
$$\gamma^{ij}(X^{t,x}_s,e)(\overline{u}_{ij}(e)-{u}_{ij}(e))\leq C
(\gamma^{ij}(X^{t,x}_s,e))^2+2(\overline{u}_{ij}(e)-{u}_{ij}(e))^2$$
then
$$\begin{array}{l}
4y_{ij}^-\times \Xi_1\times
\int_E\gamma^{ij}(X^{t,x}_s,e)(\overline{u}_{ij}(e)-{u}_{ij}(e))n(de)\\
\leq C (y^-_{ij})^2 +2\int_{{u}_{ij}-\overline{u}_{ij}<
-y_{ij}}\mathbbm{1}_{\{y^{ij}<0\}}[|(y_{ij}+u_{ij}(e)-
\overline{u}_{ij}(e))^-|^2-|y_{ij}^-|^2-2y_{ij}(u_{ij}(e)-\overline{u}_{ij}(e))]n(de)\\
\leq C (y^-_{ij})^2
+2\int_{E}\mathbbm{1}_{\{y^{ij}<0\}}[|(y_{ij}+u_{ij}(e)-
\overline{u}_{ij}(e))^-|^2-|y_{ij}^-|^2-2y_{ij}(u_{ij}(e)-\overline{u}_{ij}(e))]n(de)
\end{array}$$
which is the claim. Thus $ P-a.s.,~Y^{ij,n,m}\leq Y^{ij,n+1,m}.$ In
the same way we have also $ P-a.s.,~Y^{ij,n,m+1}\leq Y^{ij,n,m}$.\ms

The second claim is just the representation of solutions of standard
BSDEs with jumps by deterministic functions in the Markovian
framework (see \cite{guy}). The inequalities of
(\ref{v-monotonicity}) are obtained by taking $s=t$ in
(\ref{monotonicity}) in view of the representation (\ref{Feman-Kac})
of $Y^{ij,n,m}$ by $v^{ij,n,m}$ and $X^{t,x}$.
\end{proof}
\begin{rem}\label{eqvijnm}For any $\ijg$, $v^{ij,n,m}$ is the unique viscosity
solution in $\Pi_g$ of the following  integral-partial differential
equation:  \be \label{eqvijnm2}\left\{
    \begin{array}{ll}
      -\partial_tv^{ij,n,m}(t,x)-\cL
      v^{ij,n,m}(t,x)-n(v^{ij,n,m}(t,x)-\max\limits_{k\in
      A^1_i}(v^{kj,n,m}(t,x)-\underline{g}_{ik}(t,x))^-\\
      \qq +m(v^{ij,n,m}(t,x)-\min \limits_{l\in
      A^2_j}(v^{il,n,m}(t,x)+\bar {g}_{jl}(t,x))^+\\\qq-g^{ij}(t,x,(v^{kl,n,m}(t,x))_{\klg},\sigma(t,x)\tp
      D_x v^{ij,n,m}(t,x),I_{ij}(t,x,v^{ij,n,m}))=0\,;\\\\\
      v^{ij,n,m}(T,x)=h^{ij}(x).
    \end{array}
    \right.\ee
For more details one can see \cite{guy}.
\end{rem}
We now suggest two approximation schemes obtained from the sequence
$((Y^{ij,n,m})_{\ijg})_{n,m}$ of the solution of system
(\ref{approximation}). The first scheme is a sequence of decreasing
reflected BSDEs with interconnected lower obstacles defined as:
$\forall (i,j)\in A^1\times A^2$, \be \label{decreasing}\left\{
    \begin{array}{ll}
      (\bar{Y}^{ij,m},\bar{Z}^{ij,m},\bar{U}^{ij,m},\bar{K}^{ij,m})\in \cS^2\times\cH^2\times\cH^2(
      \hat{N})\times\cA^2;\\\\\
      \bar{Y}^{ij,m}_s=h^{ij}(X^{t,x}_T)+\int^T_s\bar{f}^{ij,m}(r,X^{t,x}_r,(\bar{Y}^{kl,m}_r)_{\klg},\bar{Z}^{ij,m}_r,\bar{U}^{ij,m}_r)dr-\int^T_s\bar{Z}^{ij,m}_r
      dB_r\\\qq\qq\qq
      -\int^T_s\int_E \bar{U}^{ij,m}_r(e)\hat{N}(dr,de)+\bar{K}^{ij,m}_T-\bar{K}^{ij,m}_s,~~s\leq T;\\\\\
      \bar{Y}^{ij,m}_s\geq \max\limits_{k\in
      A^1_i}\{\bar{Y}^{kj,m}_s-\underline{g}_{ik}(s,X^{t,x}_s)\},~~s\leq
      T;\\\\
      \int^T_0(\bar{Y}^{ij,m}_s- \max\limits_{k\in
      A^1_i}\{\bar{Y}^{kj,m}_s-\underline{g}_{ik}(s,X^{t,x}_s)\})d\bar{K}^{ij,m}_s=0,
    \end{array}
    \right.\ee
where for any $\ijg$, $m\geq 0$ and $(s,\vec{y},z,u)$ ($u\in
{\cL}_\R^2(E,{\cal B}_E,n)$), \be \label{eqbarf}
\begin{array}{ll}\bar{f}^{ij,m}(s,X^{t,x}_s,\overrightarrow{y},z,u)&:=g^{ij,+,m}(s,X^{t,x}_s,(y^{kl})_{(k,l)\in
A^1\times A^2},
z,\int_Eu(e)\gamma^{ij}(X^{t,x}_s,e)n(de))\\\\
&:=g^{ij}(s,X^{t,x}_s,(y^{kl})_{(k,l)\in A^1\times A^2},
z,\int_Eu(e)\gamma^{ij}(X^{t,x}_s,e)n(de))\\
&\qq\qq\qq\qq\qq\qq \qq\qq -m(y^{ij}-\min\limits_{l\in
A^2_j}(y^{il}+\overline{g}_{jl}(s,X^{t,x}_s)))^+.\end{array}\ee  The
following result is related to existence and uniqueness of the
solution of (\ref{decreasing}) and some of its properties.
\begin{propo} \label{lienedpbsde}${}$
\ms \noindent i) For any fixed $m\geq 0$, the solution
$(\bar{Y}^{ij,m},\bar{Z}^{ij,m},\bar{U}^{ij,m},\bar{K}^{ij,m})_{\ijg}$
of the system (\ref{decreasing}) exists and is unique. Moreover for
any $(i,j)$
and $m\geq 0$, we have:\\
\be\label{nlimit}\lim\limits_{n\rightarrow\infty}\E[\sup\limits_{t\leq
s\leq T}|Y^{ij,n,m}_s-\bar{Y}^{ij,m}_s|^2]\rightarrow 0 \mbox{ and
}P -a.s.,~~\bar{Y}^{ij,m}\geq \bar{Y}^{ij,m+1}.\ee (ii) There exists
a unique $m_1\times m_2$-uplet of deterministic continuous functions
$(\bar{u}^{kl,m})_{(k,l)\in A^1\times A^2}$ in $\Pi_g$ such that,
for every $t\leq T$, \be\label{mFeyman-Kac}
\bar{Y}^{ij,m}_s=\bar{u}^{ij,m}(s,X^{t,x}_s),~s\in [t,T].\ee
Moreover, for any $\ijg$ and $(t,x)\in [0,T]\times \R^k$,
$\bar{u}^{ij,m}(t,x)\geq \bar{u}^{ij,m+1}(t,x)$. Finally,
$(\bar{u}^{ij,m})_{(i,j)\in A^1\times A^2}$ is a
unique viscosity solution in the class $\Pi_g$ of the following system of
 variational inequalities with inter-connected obstacles: $\forall
\ijg$,\be
\label{mSIPDE}\left\{
    \begin{array}{ll}
      min\{\bar{u}^{ij,m}(t,x)-\max\limits_{k\in
      A^1_i}(\bar{u}^{kj,m}(t,x)-\underline{g}_{ik}(t,x));-\partial_t\bar{u}^{ij,m}(t,x)-\cL
      \bar{u}^{ij,m}(t,x)\\\qq\qq
      -g^{ij,+,m}(t,x,
      (\bar{u}^{kl,m}(t,x))_{\klg},\sigma(t,x)\tp D_x\bar{u}^{ij,m}(t,x),
      I_{ij}(t,x,\bar{u}^{ij,m})\}=0;\\
      \bar{u}^{ij,m}(T,x)=h^{ij}(x).
    \end{array}
    \right.\ee
\end{propo}
\begin{proof}(i) It is enough to consider the case $m=0$, since for
any $\ij$, the function $$(s,x,(y^{kl})_{(kl)\in A^1\times
A^2})\rightarrow -m(y^{ij}-\min\limits_{l\in
A^2_j}(y^{il}+\overline{g}_{jl}(s,x)))^+$$ has the same properties
as $f^{ij}$ displayed in ({\bf A1}) and ({\bf A2}). First and
w.l.o.g we may assume that $f^{ij}$ is non-decreasing w.r.t.
$y^{kl}$, for any $\klg$, since thanks to assumption
(A2), it is enough to multiply the solution by $e^{\lambda t}$,
where $\lambda$ is appropriately chosen in order to fall in this
latter case, since $f^{ij}$ is Lipschitz  w.r.t. the component
$y^{ij}$.
Now let $H$ (resp. $F$) be the following functions:
$$
  H(x)=\sum_{\ijg}|h^{ij}(x)| \mbox{ and
  }F(t,x,y,z,u)=\sum_{\ijg}|f^{ij}(t,x,yI_{m_1,m_2},z,u)|
$$where $(y,z,u)\in \R^{1+d}\times \cL^2_\R(E,\mathcal{B}_E,n)$ and $I_{m_1,m_2}$ is the
matrix of $m_1$ (resp. $m_2$) rows (resp. columns) with entries
equal to 1. Let $(\bar Y,\bar Z,\bar U)$ be the solution of the
following one-dimensional BSDE with jumps associated with
$(F(s,X^{t,x}_s,y,z,u),H(\xt_T))$. Next let $n$ be fixed and let us
define recursively the sequence $(\tilde{Y}^{k,ij,n})_{k\geq 0}$ as
follows: for $k=0$ and any $(i,j)\in A_1\times A_2$, we set
$\tilde{Y}^{0,ij,n}:=-\bar Y$. For $k\geq 1$, we define
$(\tilde{Y}^{k,ij,n},{Z}^{k,ij,n},{U}^{k,ij,n})\in
\cS^2\times\cH^2\times\cH^2(\hat{N})$ as the solution of the
following system of BSDEs: $\forall(i,j)\in A_1\times A_2$, \be
\label{310}\left\{
    \begin{array}{ll}
      -d\tilde{Y}^{k,ij,n}_s=f^{ij}(s,X^{t,x}_s,(\tilde{Y}^{k-1,pq,n}_s)_{(p,q)\in A^1\times A^2},\tilde{Z}^{k,ij,n}_s,\tilde{U}^{k,ij,n}_s)ds\\
     \qq\qq +n(\tilde{Y}^{k,ij,n}_s-
     \max\limits_{l\in A^1_i}(\tilde{Y}^{k-1,lj,n}_s-
     \underline{g}_{il}(s,X^{t,x}_s)))^-ds -\tilde{Z}^{k,ij,n}_sdB_s-
     \int_E \tilde{U}^{k,ij,n}_s(e)\hat{N}(ds,de),\,\, s\leq T;\\
      \tilde{Y}^{k,ij,n}_T=h^{ij}(X^{t,x}_T).
    \end{array}
    \right.\ee
The solution of \eqref{310} exists and is unique since it is a decoupled multi-dimensional standard BSDE with a Lipschiz coefficient, noting that $(\tilde{Y}^{k-1,pq,n}_s)_{(p,q)\in A^1\times A^2}$ is already given. Since $n$ is fixed and the coefficient $\phi^{ij,n}$ defined by:
$$\phi^{ij,n}(s,\omega,({y}^{pq})_{(p,q)\in A^1\times A^2},z^{ij},u^{ij}):=f^{ij}(s,X^{t,x}_s(\omega),({y}^{pq})_{(p,q)\in A^1\times A^2},z^{ij},u^{ij})
     +n(y^{ij}-\max\limits_{l\in A^1_i}(y^{lj}-\underline{g}_{il}(s,X^{t,x}_s)))^-$$
is Lipschitz w.r.t. $(({y}^{pq})_{(p,q)\in A^1\times
A^2},z^{ij},u^{ij})$, the sequence $(\tilde{Y}^{k,ij,n})_{k\geq 0}$
converges in $\cS^2$ to $Y^{ij,n,0}$ as $k\rightarrow\infty$, for
any $i,j$ and $n$. Finally by comparison and an induction argument,
used twice (with $n$ and then with $k$), we obtain:
$$
  \tilde{Y}^{k,ij,n}\leq \tilde{Y}^{k,ij,n+1} \mbox{ and }\tilde{Y}^{k,ij,n}\leq  \bar Y.
$$Note that for the second inequality, we take into account of the fact that
$n(\bar Y_s-\max\limits_{l\in A^1_i}(\bar
Y_s-\underline{g}_{il}(s,X^{t,x}_s)))^-\equiv 0$ since
$\underline g_{il}\geq 0$. Take now the limit w.r.t. $k$ in the
previous inequalities to obtain:
\be\label{estimyijno}
{Y}^{ij,n,0}\leq {Y}^{ij,n+1,0} \mbox{ and }{Y}^{ij,n,0}\leq \bar Y.
\ee
Therefore there exists a progressively measurable process $\bar Y^{ij,0}$ such that
$$Y^{ij,n,0}\nearrow\bar Y^{ij,0} \mbox{ and }\bar Y^{ij,0}\leq \bar Y.$$Now using the
monotonic limit theorem by E.Essaky (\cite{essaky}, Theorem 3.1) there exist
$(\bar Z^{ij,0},\bar U^{ij,0})\in \cH^2 \times\cH^2(\hat{N})$ and
$\bar K^{ij,0}\in {\cal S}^2$ non-decreasing such that:\\
(a) $\bar Y^{ij,0}$ belongs to ${\cal S}^2$ and for any stopping
time $\tau$, $\lim\limits_{n\rightarrow\infty}\nearrow
Y^{ij,n,0}_\tau=\bar Y^{ij,0}_\tau$.\\
(b) $\bar K^{ij,0}$ is predictable RCLL non-decreasing, $\bar
K^{ij,0}_0=0$ and for any stopping time $\tau$, the sequence
$(K^{ij,n,0}_\tau)_{n\geq 0}$ converge weakly in $L^2(P)$ to $\bar
K^{ij,0}_\tau$ ; \\(c) For any
$p\in[1,2)$,$$\lim\limits_{n\rightarrow\infty}\E[\int^T_0|Z^{ij,n,0}_s-\bar
Z^{ij,0}|^p ds]=0,
~~\lim\limits_{n\rightarrow\infty}\E[\int^T_0\int_E|U^{ij,n,0}_s-\bar
U^{ij,0}|^{\frac{p}{2}} n(de)ds]=0;$$ (d) Moreover for any $\ijg$
and $s\leq T$ we have:
 \be
\label{LimitRBSDE}\left\{
\begin{array}{ll}
      \bar Y^{ij,0}_s=h^{ij}(X^{t,x}_T)+\int^T_s
      f^{ij}(r,X^{t,x}_r,(\bar Y^{kl,0}_r)_{\klg},\bar Z^{ij,0}_r,\bar U^{ij,0}_r)dr+\bar K^{ij,0}_T-\bar K^{ij,0}_s\\
      \qq\qq\qq\qq\qq\qq\qq\qq\qq-\int^T_s\bar Z^{ij,0}_r
      dB_r-\int^T_s\int_E \bar U^{ij,0}_r(e) n(de)dr \,\,;\\
      \bar Y^{ij,0}_s\geq \max\limits_{k\in
      A^1_i}\{\bar Y^{kj,0}_{s}-\underline{g}_{ik}(s,X^{t,x}_{s})\};\\
      \int_0^T(\bar Y^{ij,0}_{s-}- \max\limits_{k\in
      A^1_i}\{\bar Y^{kj,0}_{s-}-\underline{g}_{ik}({s},X^{t,x}_{s-}))d\bar
      K^{ij,0}_s=0.      
   \end{array}
    \right.\ee
The remaining of the proof is the same as the one of Theorem 3.2 in
\cite{zhao}, pp.1623, i.e.,  to show that the predictable process
$\bar K^{ij,0}$ is continuous thanks to the non free loop property
({\bf A4}). Thus $(\bar Y^{ij,0},\bar Z^{ij,0},\bar U^{ij,0},\bar
K^{ij,0})$ is a solution of (\ref{decreasing}) with $m=0$.\\
Uniqueness of the solution of (\ref{decreasing}) is obtained in the same way as in
(\cite{hamadenemorlais}, pp.193) or
(\cite{eliekharroubichassegneux}, pp.122) in making the connection
between the solutions of systems of type (\ref{decreasing}) and the
value function of the related optimal switching problem. This is possible since the hypotheses on the data allow for comparison in this framework of Brownian-Poisson noise type (especially (A1)). \\ Finally the
last property of convergence stems from the following facts: (i)
$Y^{ij,n,0}\nearrow _n \bar Y^{ij,0}$; (ii)
$Y_-^{ij,n,0}\nearrow_n\bar Y^{ij,0}_-$; (iii) A weak version of
Dini's theorem for RCLL process (see \cite{dlm}, pp.202). Note that
property (ii) is a consequence of continuity of $\bar K^{ij,0}$
which implies that the predictable projection of $\bar Y^{ij,0}$ is
nothing but $\bar Y^{ij,0}_-$ and the same holds for $Y^{ij,n,0}$.\ms

\no (ii) By (\ref{Feman-Kac}), (\ref{v-monotonicity}) and
(\ref{estimyijno}), we obtain that the sequence of functions $(v^{ij,n,0})_{n\geq 0}$ is convergent for any $\ijg$. So let us set $\bar{u}^{ij,0}(t,x):=\lim_n\nearrow v^{ij,n,0}(t,x)$. Therefore by (\ref{Feman-Kac}) and (\ref{nlimit}), the relation (\ref{mFeyman-Kac}) holds true.

Next as previously mentionned, we can obtain the same results for arbitrary $m$ and not only for $m=0$. Therefore we define $\bar u^{ij,m}(t,x):=\lim_n v^{ij,n,m}(t,x)$. Those functions verify (\ref{mFeyman-Kac}) and
$\bar u^{ij,m}\geq  \bar u^{ij,m+1}$ by  (\ref{v-monotonicity}). Next $(\bar{u}^{ij,m})_{\ij}$ is a viscosity solution of (\ref{mSIPDE})
since $(v^{ij,n,m})_{\ijg}$ is solution of (\ref{eqvijnm}) and by
arguing as in (\cite{zhao}, Theorem 4.1). Finally uniqueness in the
class $\Pi_g$ and continuity holds true by Remark
\ref{unicitecasunebarriere}.
\end{proof}
We now consider the increasing approximating scheme: $\forall \ijg$, \be\label{increasing}\left\{
    \begin{array}{ll}
      (\underline{Y}^{ij,n},\underline{Z}^{ij,n},\underline{U}^{ij,n},\underline{K}^{ij,n})\in \cS^2\times\cH^2\times\cH^2
      (\hat{N})\times\cA^2;\\\\\
      \underline{Y}^{ij,n}_s=h^{ij}(X^{t,x}_T)+\int^T_s\underline{f}^{ij,n}(r,X^{t,x}_r,(\underline{Y}^{k,l,n}_r)_{\klg},\underline{Z}^{ij,n}_r,\underline{U}^{ij,n}_r)dr-\int^T_s\underline{Z}^{ij,n}_rdB_r\\\
      \qq\qq -\int^T_s\int_E \underline{U}^{ij,n}_r(e)\hat{N}(dr,de)-(
      \underline{K}^{ij,n}_T-\underline{K}^{ij,n}_s),~~s\leq T;\\\\\
      \underline{Y}^{ij,n}_s\leq \min\limits_{l\in
      A^2_j}\{\underline{Y}^{il,n}_s+\overline{g}_{jl}(s,X^{t,x}_s)\},~~s\leq
      T,\\\\\
      \int^T_0(\underline{Y}^{ij,n}_s- \min\limits_{l\in
      A^2_j}\{\underline{Y}^{kj,n}_s+\overline{g}_{jl}(s,X^{t,x}_s)\})d\underline{K}^{ij,n}_s=0,
    \end{array}
    \right.\ee
where for any $\ijg$, $n\geq 0$ and
$(s,\vec{y},z,u)$,
\begin{align*}\underline{f}^{ij,n}(s,X^{t,x}_s,\vec{y},z,u):=&g^{ij,-,n}(s,X^{t,x}_s,(y^{kl})_{(k,l)\in A^1\times A^2},
z,\pin u(e)\gamma^{ij}(X^{t,x}_s,e)n(de))\\\\
:=&g^{ij}(s,X^{t,x}_s,(y^{kl})_{\klg},
z,\mbox{$\int_E$} u(e)\gamma^{ij}(X^{t,x}_s,e)n(de))\\
&\qq\qq\qq\qq\qq\qq +n(y^{ij}-\max\limits_{k\in
A^1_i}(y^{kj}-\underline{g}_{ik}(s,X^{t,x}_s)))^-.
\end{align*}
The existence of
$(\underline{Y}^{ij,n},\underline{Z}^{ij,n},\underline{U}^{ij,n},\underline{K}^{ij,n})_{\ij}$
is obtained thanks to Proposition \ref{lienedpbsde} in considering the system of
reflected BSDEs with interconnected lower obstacles associated with the
data \\$\{(-\underline{f}^{ij,n}(s,X^{t,x}_s,-\overrightarrow{y},-z,-u))_{\ij},
 (-h^{ij})_{\ij}, (\bar g_{jl})_{j,l\in A_2}\}$ which has a unique solution\\
 $({Y_1}^{ij,n},{Z_1}^{ij,n},{U_1}^{ij,n},{K_1}^{ij,n})_{\ij}$ and then it is enough to set  $(\underline{Y}^{ij,n},\underline{Z}^{ij,n},\underline{U}^{ij,n},\underline{K}^{ij,n})_{\ij}=(-{Y_1}^{ij,n},-{Z_1}^{ij,n},-{U_1}^{ij,n},
 {K_1}^{ij,n})_{\ij}$.
 The following is the analogous of Proposition \ref{lienedpbsde}.
 \begin{propo}\label{lienedpbsde2}
 \noindent i) For any fixed $\ijg$ and $n\geq 0$ we have:
\be\label{nlimit2}\lim\limits_{m\rightarrow\infty}\E[\sup\limits_{t\leq
s\leq T}|Y^{ij,n,m}_s-\underline {Y}^{ij,n}_s|^2]\rightarrow 0
\mbox{ and }P -a.s.,~~\underline{Y}^{ij,n}\leq
\underline{Y}^{ij,n+1}.\ee (ii) There exists a unique $m_1\times
m_2$-uplet of deterministic continuous functions
$(\underline{u}^{kl,n})_{(k,l)\in A^1\times A^2}$ in $\Pi_g$ such
that, for every $\tx$,\be\label{nFeyman-Kac} \underline {
Y}^{ij,n}_s=\underline{u}^{ij,n}(s,X^{t,x}_s),\forall s\in [t,T].\ee
Moreover, for any $\ijg$ and $(t,x)\in [0,T]\times \R^k$,
$\underline{u}^{ij,n}(t,x)\leq \underline{u}^{ij,n+1}(t,x)$.
Finally, $(\underline{u}^{ij,n})_{(i,j)\in A^1\times A^2}$ is the
unique viscosity solution in the class $\Pi_g$ of the following
system of
 variational inequalities with inter-connected upper obstacles. $\forall
 \ijg$,\be
\label{nSIPDE}\left\{
    \begin{array}{ll}
      \max\{\underline{u}^{ij,n}(t,x)-\min\limits_{l\in
      A^2_j}(\underline{u}^{il,n}(t,x)+\overline{g}_{jl}(t,x));
      -\partial_t\underline{u}^{ij,n}(t,x)-\cL
      \underline{u}^{ij,n}(t,x)\\\
      \qq\qq -g^{ij,-,n}(t,x,(\underline{u}^{kl,n}(t,x))_{\klg},\sigma(t,x)\tp D_x\underline{u}^{ij,n}(t,x),
      I_{ij}(t,x,\underline{u}^{ij,n}))\}=0~\,\,;\\\
      \underline{u}^{ij,n}(T,x)=h^{ij}(x).
    \end{array}
    \right.\ee
\end{propo}
\no Now let us define $\bar{u}^{ij}$, $\underline{u}^{ij}$, $\ijg$, by:
$$\bar{u}^{ij}(t,x):=\lim\limits_{m\rightarrow\infty}\bar{u}^{ij,m}(t,x),~~\underline{u}^{ij}(t,x):=\lim\limits_{n\rightarrow\infty}\underline{u}^{ij,n}(t,x),\,\,\tx.$$
We then have:
\begin{cor} For any $(i,j)\in A^1\times A^2$, the function $\bar{u}^{ij}$ (resp.
$\underline{u}^{ij}$) is $usc$ (resp. $lsc$). Moreover,
$\bar{u}^{ij}$ and $\underline{u}^{ij}$ belong to $\Pi_g$ and for any
$(t,x)\in[0,T]\times \R^k$, $$ \underline{u}^{ij}(t,x)\leq
\bar{u}^{ij}(t,x).$$
\end{cor}
\begin{proof}
For any $(i,j)\in A_1\times A_2$, the function $\bar{u}^{ij}$ (resp.
$\underline{u}^{ij}$) is obtained as a decreasing
 (resp. increasing) limit of continuous functions. Therefore, it is
$usc$  (resp. $lsc$). Next, for any $(i,j)$ and $n$, $m$,
$$u^{ij,n,m}(t,x)\leq u^{ij,n,0}(t,x),~~(t,x)\in [0,T]\times \R^k,$$ as the sequence $(u^{ij,n,m})_{m\geq 0}$ is decreasing. Thus, taking the limit as $m\rightarrow\infty$ we obtain, $$\underline{u}^{ij,n}\leq u^{ij,n,0}.$$
Now using \eqref{Feman-Kac} and \eqref{nlimit}, it follows that, for any $t\leq T$ and $s\in [t,T]$, $Y^{ij,n,0}_s=u^{ij,n,0}(s,X^{t,x}_s)$ and the processes $Y^{ij,n,0}$ converges in $\cS^2$, as $n\rightarrow\infty$, to $\bar{Y}^{ij,0}$ which is solution of \eqref{decreasing} with $m=0$.
 Furthermore, by \eqref{mFeyman-Kac}, there exists a deterministic continuous function $\bar{u}^{ij,0}$ with polynomial growth such that for any $t\leq T$ and $s\in [t,T]$, $\bar Y^{ij,0}_s=\bar{u}^{ij,0}(s,X^{t,x}_s)$. Then taking $s=t$ and the limit as $n\rightarrow\infty$  to obtain
$$\underline{u}^{ij}(t,x):=\lim\limits_{n\rightarrow\infty}\underline{u}^{ij,n}(t,x)\leq \lim\limits_{n\rightarrow\infty}u^{ij,n,0}(t,x)=\bar{u}^{ij,0}(t,x),~~\forall(t,x)\in [0,T]\times \R^k.$$
But $\bar{u}^{ij,0}$ and $\underline{u}^{ij,n}$ belong to $\Pi_g$ and $\underline{u}^{ij,n}\leq \underline{u}^{ij,n+1}$. Thus $\underline{u}^{ij}\in \Pi_g$, for any $(i,j)\in A_1\times A_2$. The last inequality follows from \eqref{v-monotonicity} and the definitions of $\bar{u}^{ij}$ and $\underline{u}^{ij}$. In the  same way one can show that $\bar{u}^{ij}\in \Pi_g$, for any $(i,j)\in A_1\times A_2$.
\end{proof}
We now focus on the proof of existence of a solution for system (\ref{SIPDE}).
\begin{propo}\label{proprietedesoussolution}
The family $(\bar{u}^{ij})_{(i,j)\in A^1\times A^2}$ is a viscosity
subsolution of the system (\ref{SIPDE}).
\end{propo}
\begin{proof} First recall that for any $\ijg$, $\bar{u}^{ij}=\lim_{m}\searrow
\bar u^{ij,m}$, so that $\bar{u}^{ij}$ is $usc$. Moreover since $\bar{u}^{ij,m}(T,x)=h^{ij}(x)$ then $\bar{u}^{ij}(T,x)=h^{ij}(x),\,\forall x\in \R^k.$

Now let $\ijg$ and
$(t,x)\in (0,T)\times \R^k$ be fixed. We suppose that there exists
$\epsilon_0>0$ s.t.\be \label{ineqzz}\bar{u}^{ij}(t,x)\geq
L^{ij}[\vec{\bar{u}}](t,x)+\epsilon_0,\ee otherwise the subsolution
property holds. Thanks to the decreasing convergence of
$(\bar{u}^{ij,m})_{m\geq 0}$ to $\bar{u}^{ij}$, there exists $m_0$
such that for any $m\geq m_0$, we have
\be\label{5210}\bar{u}^{ij,m}(t,x)\geq
L^{ij}[(\bar{u}^{pq,m})_{(p,q)\in A^1\times
A^2}](t,x)+\frac{\epsilon_0}{2}.\ee As for any $m\geq 0$,
$\bar{u}^{ij,m}$ and $L^{ij}[(\bar{u}^{pq,m})_{(p,q)\in A^1\times
A^2}]$ are continuous, then there exists a neighborhood $\Theta_m$ of
$(t,x)$ such that \be\label{07031}\bar{u}^{ij,m}(t',x')\geq
L^{ij}[(\bar{u}^{pq,m})_{(p,q)\in A^1\times
A^2}](t',x')+\frac{\epsilon_0}{4},\,\forall (t',x')\in \Theta_m.\ee

Let now $\phi$ be a ${\cal C}^{1,2}$-function such that
$\phi(t,x)=\bar u^{ij}(t,x)$ and $\bar u^{ij}-\phi $ has a global
strict maximum in $(t,x)$. Next let $\d>0$ and for $m\geq 0$ let
$(t_m,x_m)$ be the global maximum of $\bar u^{ij}-\phi$ on
$[0,T]\times B'(x,2\d K_\b)$ ($K_\b$ is a bound for $\beta$ and
$B'(x,2\d K_\b)$ is the closure of $B(x,2\d K_\b)$), which exists
since the function $\bar u^{ij}-\phi$ is $usc$. But there exists a
subsequence $\{m_k\}$ such that \be
\label{suiteextraite}(t_{m_k},x_{m_k})\rw_k(t,x) \mbox{ and }\bar
u^{ij, m_k}(t_{m_k},x_{m_k})\rw_k \bar{u}^{ij}(t,x).\ee Actually by
Lemma 6.1 in \cite{usersguide}, there exist a subsequence $\{m_k\}$
and a sequence $(t'_{m_k},x'_{m_k})_k$ such that $$
(t'_{m_k},x'_{m_k})_k\rightarrow_k (t,x) \mbox{ and }\bar
u^{ij,m_k}(t'_{m_k},x'_{m_k})\rightarrow \bar u^{ij}(t,x).
$$Next let us consider a convergent subsequent of $(t_{m_k},x_{m_k})$,
which we still denote by $(t_{m_k},x_{m_k})$, and let $(\bar t,\bar
x)$ be its limit. Then for some $k_0$ and for $k\geq k_0$ we have
$$\begin{array}{l}\bar{u}^{ij}(\bar t ,\bar x)-\phi(\bar t,\bar
x)\geq
\limsup_k(\bar{u}^{ij,m_k}(t_{m_k},x_{m_k})-\phi(t_{m_k},x_{m_k}))\geq
\liminf_k(\bar{u}^{ij,m_k}(t_{m_k},x_{m_k})-\phi(t_{m_k},x_{m_k}))
\\\qq \geq \liminf_k (
\bar{u}^{ij,m_k}(t'_{m_k},x'_{m_k})-\phi(t'_{m_k},x'_{m_k}))=
\bar{u}^{ij}(t,x)-\phi(t,x)\geq \bar{u}^{ij}(\bar t ,\bar
x)-\phi(\bar t,\bar x). \end{array}$$ It implies that
$\bar{u}^{ij}(t,x)-\phi(t,x)=\bar{u}^{ij}(\bar t ,\bar x)-\phi(\bar
t,\bar x)$ then $(t,x)=(\bar t,\bar x)$ since the maximum is strict.
On the other hand we obviously have $\bar u^{ij,
m_k}(t_{m_k},x_{m_k})\rw_k \bar{u}^{ij}(t,x)$. Finally since this is
valid for any subsequence of $(t_{m_k},x_{m_k})$, then the claim
follows.

But from the subsequence $\{m_k\}$ one can substract a subsequence
which we still denote by $\{m_k\}$ such that $(t_{m_k},x_{m_k})$
belongs to $\Theta_{m_k}$. Indeed if this is not possible one can
find a subsequence $\{m_p\}$ of $\{m_k\}$ such that for $p\geq 0$,
$(t_{m_p},x_{m_p})$ does not belong to $\Theta_{m_p}$, i.e.,
$$\bar{u}^{ij,m_p}(t_{m_p},x_{m_p})<L^{ij}[(\bar{u}^{pq,m_p})_{(p,q)\in
A^1\times A^2}](t_{m_p},x_{m_p})+\frac{\epsilon_0}{2}.$$ Then in taking the limit w.r.t. $p$ we obtain
$$\bar{u}^{ij}(t,x)\leq (L^{ij}[(\bar{u}^{pq})_{(p,q)\in
A^1\times A^2}])^*(t,x)+\frac{\epsilon_0}{2}$$where $(.)^*$ stands
for the upper semi-continuous envelope. But
$$
(L^{ij}[(\bar{u}^{pq})_{(p,q)\in A^1\times A^2}])^*=\max_{k\in
A_1^i}(\bar{u}^{kj}-\underline g_{ik})^*=\max_{k\in
A_1^i}(\bar{u}^{kj}-\underline g_{ik}).
$$ Therefore we have
       $$\bar{u}^{ij}(t,x)\leq \max_{k\in
A_1^i}(\bar{u}^{kj}-\underline
g_{ik})(t,x)+\frac{\epsilon_0}{2}=L^{ij}[\bar{u}](t,x)+\frac{\epsilon_0}{2}
   $$
which is contradictory with (\ref{ineqzz}). Hereafter we consider
this subsequence $\{m_k\}$.

Now for $k$ large enough: (i) $(t_{m_k},x_{m_k})\in (0,T)\times
B(x,2K_\b\delta)$ and is the global maximum of
$\bar{u}^{ij,m_k}-\phi$ in $(0,T)\times B(x_{m_k},K_\b\delta)$ ;
(ii)
$\bar{u}^{ij,m_k}(t_{m_k},x_{m_k})>L^{ij}[(\bar{u}^{pq,m_k})_{(p,q)\in
A^1\times A^2}](t_{m_k},x_{m_k})$. \ms As $\bar{u}^{ij,m_k}$ is a
subsolution of (\ref{mSIPDE}), then by Proposition 5.1 - Remark
\ref{remappendix} in Appendix, we have
\be\label{eqlim2}\begin{array}{l}
 ~~~-\partial_t\phi(t_{m_k},x_{m_k})-\bar{{\cal L}}
    \phi(t_{m_k},x_{m_k})-I^{1}_\d(t_{m_k},x_{m_k},\phi)-
    I^{2}_\d(t_{m_k},x_{m_k},D_x\phi(t_{m_k},x_{m_k}),
    \bar{u}^{ij,m_k})\\\qq \qq \qq +m_k(\bar u^{ij,m_k}
      (t_{m_k},x_{m_k})-\min_{l\in A_2^j}(\bar u^{il,m_k}
      (t_{m_k},x_{m_k})+\bar g_{jl}(t_{m_k},x_{m_k})))^+
          \\\
 \leq g^{ij}[t_{m_k},x_{m_k},(\bar{u}^{pl,m_k}(t_{m_k},x_{m_k}))_{(p,l)\in
      A^1\times
      A^2},\sigma (t_{m_k},x_{m_k})\tp D_x\phi(t_{m_k},x_{m_k}),\\
      \qq\qq\qq  I^{1,\d}_{ij}(t_{m_k},x_{m_k},\phi)+I^{2,\d}_{ij}(t_{m_k},x_{m_k},\bar{u}^{ij,m_k})].\end{array}
\ee From which we deduce, in dividing both hand-sides of (\ref{eqlim2}) by $m_k$ and then taking the limit as $k\rightarrow
\infty$,  that
$$\eps_k=(\bar u^{ij,m_k}
      (t_{m_k},x_{m_k})-\min_{l\in A_2^j}(\bar u^{il,m_k}
      (t_{m_k},x_{m_k})+\bar g_{jl}(t_{m_k},x_{m_k})))^+\rightarrow
      _k0.$$
      Next fix $k_0$ and let $k\geq k_0$. As the sequence
      $(\bar u^{ij,m})_m$ is decreasing then
      $$\begin{array}{ll}
      \bar u^{ij,m_k}(t_{m_k},x_{m_k}) &\leq
            \min_{l\in A_2^j}(\bar u^{il,m_k}
      (t_{m_k},x_{m_k})+\bar g_{jl}(t_{m_k},x_{m_k}))+\eps_k \\
      {}&\leq \min_{l\in A_2^j}(\bar u^{il,m_{k_0}}
      (t_{m_k},x_{m_k})+\bar g_{jl}(t_{m_k},x_{m_k}))+\eps_k
      \end{array}
      $$
Take the limit w.r.t $k$, using continuity of $\bar
u^{il,m_{k_0}}$ then send $k_0$ to $+\infty$ to obtain:
$$
  \bar u^{ij}(t,x)\leq
            \min_{l\in A_2^j}(\bar u^{il}(t,x)+\bar g_{jl}(t,x)).
$$
Next there exists a subsequence of $\{m_k\}$ (which we still denote
by $\{m_k\}$) such that: 

(i) $\forall (p,l)\in A^1_i\times A^2_j$,
$(\bar{u}^{pl,m_k}(t_{m_k},x_{m_k}))_k$ is convergent and then
$\lim_k\bar{u}^{pl,m_k}(t_{m_k},x_{m_k})\leq \bar{u}^{pl}(t,x)$ ;

(ii) $I^{1}_\d(t_{m_k},x_{m_k},\phi)\rw_k I^{1}_\d(t,x,\phi)$ and 
$I^{1,\d}_{ij}(t_{m_k},x_{m_k},\phi))\rw_k I^{1,\d}_{ij}(t,x,\phi)$;

(iii) By Fatou's Lemma, $\limsup_k I^{2}_\d(t_{m_k},x_{m_k},D_x\phi(t_{m_k},x_{m_k}),\bar{u}^{ij,m_k})\leq I^{2}_\d(t,x,D_x\phi(t,x),\bar{u}^{ij})$  and  \\
$\limsup_k I^{2,\d}_{ij}(t_{m_k},x_{m_k},\bar{u}^{ij,m_k})))\leq I^{2,\d}_{ij}(t,x,\bar{u}^{ij}).$ 
\ms 

\no Let us now set: $$\begin{array}{l}\Delta_k:=
 g^{ij}[t_{m_k},x_{m_k},(\bar{u}^{pl,m_k}(t_{m_k},x_{m_k}))_{(p,l)\in
      A^1\times
      A^2},\sigma (t_{m_k},x_{m_k})\tp D_x\phi(t_{m_k},x_{m_k}),\\
    \qq  \qq  \qq  \qq  \qq I^{1,\d}_{ij}(t_{m_k},x_{m_k},\phi)+I^{2,\d}_{ij}(t_{m_k},x_{m_k},\bar{u}^{ij,m_k})]\\
- g^{ij}[t_{m_k},x_{m_k},(\bar{u}^{pl,m_k}(t_{m_k},x_{m_k}))_{(pl)\in
      A^1\times
      A^2},\sigma (t_{m_k},x_{m_k})\tp D_x\phi (t_{m_k},x_{m_k}),
      I^{1,\d}_{ij}(t,x,\phi)+I^{2,\d}_{ij}(t,x,\bar{u}^{ij})]
\end{array}
$$
Then, by linearizing $g^{ij}$ w.r.t. $q$, there exists a non-negative bounded quantity $\Xi_2$ such that 
\begin{align*}
\Delta_k&=\Xi_2 \times (I^{1,\d}_{ij}(t_{m_k},x_{m_k},\phi)+I^{2,\d}_{ij}(t_{m_k},x_{m_k},\bar{u}^{ij,m_k})-I^{1,\d}_{ij}(t,x,\phi)-I^{2,\d}_{ij}(t,x,\bar{u}^{ij})\\
&\leq C_{ij} \times (I^{1,\d}_{ij}(t_{m_k},x_{m_k},\phi)+I^{2,\d}_{ij}(t_{m_k},x_{m_k},\bar{u}^{ij,m_k})-I^{1,\d}_{ij}(t,x,\phi)-I^{2,\d}_{ij}(t,x,\bar{u}^{ij}))^+
\end{align*}
where $C_{ij}$ is the Lipschitz constant of $g^{ij}$. Therefore, with (ii)-(iii) above, we have that $\limsup_k\Delta_k\leq 0$. 
\ms

\no Going back now to (\ref{eqlim2}), and take the limit superior w.r.t. $k$ to get:  
$$\begin{array}{l}
 ~~~-\partial_t\phi(t,x)-\bar{{\cal L}}
    \phi(t,x)\leq I^{1}_\d(t,x,\phi)+I^{2}_\d(t,x,D_x\phi(t,x),\bar{u}^{ij})+\\\qq\qq\qq\qq g^{ij}[t,x,(\bar{u}^{pl}(t,x))_{(p,l)\in \G},\sigma(t,x)^\top D_x\phi(t,x),I^{1,\d}_{ij}(t,x,\phi)+I^{2,\d}_{ij}(t,x,\bar{u}^{ij})].\end{array}
$$
But $\bar{u}^{ij}(t,x)=\phi(t,x)$ and $\bar{u}^{ij}\leq \phi$, then
$I^{2}_\d(t,x,D_x\phi(t,x),\bar{u}^{ij})\leq
I^{2}_\d(t,x,D_x\phi(t,x),\phi)$ and $I^{2,\d}_{ij}(t,x,\bar{u}^{ij})\leq
I^{2,\d}_{ij}(t,x,\phi)$ . Plugging now this inequality
in the previous one to obtain
$$\begin{array}{l}
 -\partial_t\phi(t,x)-\bar{{\cal L}}
    \phi(t,x)-I(t,x,\phi)-
g^{ij}[t,x,(\bar{u}^{pl}(t,x))_{(p,l)\in \G},\sigma(t,x)^\top D_x\bar{u}^{ij}(t,x),I_{ij}(t,x,\phi)]
\leq 0.\end{array}
$$ Therefore $\bar{u}^{ij}$ is a viscosity subsolution of
$$\left\{
    \begin{array}{ll}
\min\{(w-L^{ij}[({\bar u }^{kl})_{\klg }])(t,x); \max\{(w-U^{ij}[({\bar u }^{kl})_{\klg }])(t,x);\\\
-\partial_t w (t,x)-\cL
w(t,x)-g^{ij}(t,x,[({\bar u }^{pl}(t,x))_{(p,l)\in \G^{-(i,j)}},w],\sigma(t,x)^\top D_x w(t,x),I_{ij}(t,x,w))\}\}=0;\\\
     w(T,x)=h^{ij}(x).
    \end{array}
    \right.$$
As $(i,j)$ in $\Gamma$ is arbitrary then $(\bar{u}^{ij})_{(i,j)\in A^1\times A^2}$
is a viscosity subsolution of (\ref{SIPDE}).
\end{proof}
\begin{propo}
Let $m_0$ be fixed in $\mathbf{N}$. Then  the family
$(\bar{u}^{ij,m_0})_{(i,j)\in A^1\times A^2}$ is a viscosity
supersolution of system (\ref{SIPDE}).
\end{propo}
\begin{proof}

Recall that
$(\bar{Y}^{ij,m_0},\bar{Z}^{ij,m_0},\bar{U}^{ij,m_0},\bar{K}^{ij,m_0})_{(i,j)\in
A^1\times A^2}$ solves the system of reflected BSDEs
(\ref{decreasing}). Therefore if we set $\bar K^{ij,m_0,-}_s:=
m_0\int_0^s(Y^{ij,m_0}_r-\min\limits_{l\in
A^2_j}(Y^{il,m_0}_r+\overline{g}_{jl}(r,X^{t,x}_r)))^+dr$, $s\leq T$,
then \\
$(\bar{Y}^{ij,m_0},\bar{Z}^{ij,m_0},\bar{U}^{ij,m_0},\bar{K}^{ij,m_0},\bar
K^{ij,m_0,-})_{(i,j)\in A^1\times A^2}$ is a solution of the
following system of reflected BSDEs with bilateral interconnected
obstacles: for any $\ijg$ and $s\leq T$, \be
\label{2decreasing}\left\{
    \begin{array}{ll}\bar{Y}^{ij,m_0}_s=h^{ij}(X^{t,x}_T)+
    \int^T_s{f}^{ij}(r,X^{t,x}_r,(\bar{Y}^{kl,m_0}_r)_{\klg},
    \bar{Z}^{ij,m_0}_r,\bar{U}^{ij,m_0}_r)dr-\int^T_s\bar{Z}^{ij,m_0}_rdB_r\\\qq\qq\qq\qq\qq\qq
      -\int^T_s\int_E \bar{U}^{ij,m_0}_r(e)\hat{N}(dr,de)+
      \bar{K}^{ij,m_0}_T-\bar{K}^{ij,m_0}_s-(\bar{K}^{ij,m_0,-}_T-\bar{K}^{ij,m_0,-}_s);\\
      \max\limits_{k\in
      A^1_i}\{\bar{Y}^{kj,m_0}_s-\underline{g}_{ik}(s,X^{t,x}_s)\}\leq \bar{Y}^{ij,m_0}_s\leq
      \bar{Y}^{ij,m_0}_s\vee \min\limits_{l\in
A^2_j}(\bar Y^{il,m_0}_s+\overline{g}_{jl}(s,X^{t,x}_s));\\
      \int^T_0(\bar{Y}^{ij,m_0}_s- \max\limits_{k\in
      A^1_i}\{\bar{Y}^{kj,m_0}_s-\underline{g}_{ik}(s,X^{t,x}_s)\})d\bar{K}^{ij,m_0}_s=
      \\\qq\qq\qq\qq
       \int_0^T(\bar{Y}^{ij,m_0}_s-\bar{Y}^{ij,m_0}_s\vee \min\limits_{l\in
A^2_j}(\bar Y^{il,m_0}_s+\overline{g}_{jl}(s,X^{t,x}_s)))dK^{ij,m_0,-}_r=0
    \end{array}
    \right.\ee On the other hand we know by (\ref{mFeyman-Kac}) that there exist deterministic continuous functions
$(\bar{u}^{ij,m_0})_{(i,j)\in A^1\times A^2}$ in $\Pi_g$ such that,
for every $t\leq T$, $$
\bar{Y}^{ij,m_0}_s=\bar{u}^{ij,m_0}(s,X^{t,x}_s),\,s\in [t,T].$$ Then using a result by Harraj et al. \cite{hot} we deduce that for any $\ijg$, $\bar{u}^{ij,m_0}$ is a viscosity solution of the following IPDE:                        $$\left\{
    \begin{array}{ll}
\min\{\bar u^{ij,m_0}(t,x)-\max_{k\in A^1_i}\{\bar{u}^{kj,m_0}(t,x)-\underline{g}_{ik}(t,x)\};\\\
\max\{\bar u^{ij,m_0}(t,x)-\bar{u}^{ij,m_0}(t,x)\vee \min_{l\in
A^2_j}(\bar{u}^{il,m_0}(t,x)+ \bar{g}_{jl}(t,x));-\partial_t \bar
u^{ij,m_0}(t,x)-\cL \bar
u^{ij,m_0}(t,x)\\
\qquad\qquad\qquad -g^{ij}(t,x,(\bar{u}^{kl,m_0}(t,x))_{(k,l)\in
A^1\times A^2},\sigma(t,x)^\top D_x \bar u^{ij,m_0}(t,x),I_{ij}(t,x,\bar u^{ij,m_0}))\}\}=0;\\\
    \bar u^{ij,m_0}(T,x)=h^{ij}(x).
\end{array}
    \right.$$
But
$$\bar u^{ij,m_0}(t,x)-\bar{u}^{ij,m_0}(t,x)\vee min_{l\in
A^2_j}(\bar{u}^{il,m_0}(t,x)+ \bar{g}_{jl}(t,x))\leq \bar
u^{ij,m_0}(t,x)-min_{l\in A^2_j}(\bar{u}^{il,m_0}(t,x)+
\bar{g}_{jl}(t,x))
$$
Therefore $\bar{u}^{ij,m_0}$ is a supersolution of 
$$\left\{
    \begin{array}{ll}
\min\{\bar u^{ij,m_0}(t,x)-\max_{k\in A^1_i}\{\bar{u}^{kj,m_0}(t,x)-\underline{g}_{ik}(t,x)\};\\\
\max\{\bar u^{ij,m_0}(t,x)-\min_{l\in
A^2_j}(\bar{u}^{il,m_0}(t,x)+ \bar{g}_{jl}(t,x));-\partial_t \bar
u^{ij,m_0}(t,x)-\cL \bar
u^{ij,m_0}(t,x)\\
\qquad\qquad\qquad -g^{ij}(t,x,(\bar{u}^{kl,m_0}(t,x))_{(k,l)\in
A^1\times A^2},\sigma(t,x)^\top D_x \bar u^{ij,m_0}(t,x),I_{ij}(t,x,\bar u^{ij,m_0}))\}\}=0;\\\
    \bar u^{ij,m_0}(T,x)=h^{ij}(x).
\end{array}
    \right.$$
As $(i,j)$ is arbiratry then $(\bar{u}^{ij,m_0})_{(i,j)\in A^1\times A^2}$ is a viscosity supersolution of
system (\ref{SIPDE}). 
\end{proof}

Consider now the set $\mathcal{U}_{m_0}$ defined as follows:
$$\mathcal{U}_{m_0}=\{\vec{u}:=(u^{ij})_{(i,j)\in
A^1\times A^2} \mbox{ s.t. } \vec{u} \mbox{ is a subsolution of
}(\ref{SIPDE})\mbox{ and }\forall (i,j)\in A^1\times A^2,
\bar{u}^{ij}\leq u^{ij}\leq \bar{u}^{ij,m_0}\}.$$
$\mathcal{U}_{m_0}$ is not empty since it contains
$(\bar{u}^{ij})_{(i,j)\in A^1\times A^2}$. Next for $(t,x)\in
[0,T]\times \R^k$ and $\ijg$, let us
set:$$^{m_0}u^{ij}(t,x)=\sup\{u^{ij}(t,x),~(u^{kl})_{(k,l)\in
A^1\times A^2}\in \mathcal{U}_{m_0}\}.$$

We are now ready to give the main result of this
paper:
\begin{thm}The family $(^{m_0}u^{ij})_{(i,j)\in A^1\times
A^2}$ does not depend on $m_0$ and is the unique continuous
viscosity solution in the class $\Pi_g$ of the system (\ref{SIPDE}).
Moreover $^{m_0}u^{ij}=\bar u^{ij}$ for any $(i,j)\in A^1\times
A^2$.
\end{thm}
\begin{proof} Firs note that w.l.o.g we assume that for any $\ij$, the function \\$y\in
\R \mapsto g_{ij}(t,x,[(y^{kl})_{(k,l)\in \G^{-(i,j)}},y],z,q)$ is also
non-decreasing when the other variables are fixed. \ms

To begin with, note that for any $(i,j)\in A^1\times A^2$,
$\bar{u}^{ij}\leq {}^{m_0}u^{ij} \leq \bar{u}^{ij,m_0}$. Since
$\bar{u}^{ij}$ and $\bar{u}^{ij,m_0}$ are of polynomial growth, then
$(^{m_0}u^{ij})_{(i,j)\in A^1\times A^2}$ belongs also to $\Pi_g$.
The remaining of the proof is divided into two steps and for ease of
notation, we denote $(^{m_0}u^{ij})_{\ijg}$ simply by
$(u^{ij})_{\ijg}$ as no confusion is possible. \ms

\no {\bf Step 1:} First we show that $(u^{ij})_{(i,j)\in A^1\times
A^2}$ is a subsolution of (\ref{SIPDE}). As $\bar{u}^{ij}\leq
u^{ij}\leq \bar{u}^{ij,m_0}$ then $\bar{u}^{ij}\leq u^{ij,*}\leq
\bar{u}^{ij,m_0}$ since $\bar{u}^{ij}$ is $usc$ and
$\bar{u}^{ij,m_0}$ is continuous. Therefore, for any $x\in \R^k$,
since $\bar{u}^{ij}(T,x)=\bar{u}^{ij,m_0}(T,x)=h^{ij}(x)$, we have
$u^{ij,*}(T,x)=h^{ij}(x)$.\\

Next let $(\tilde{u}^{ij})_{(i,j)\in A^1\times A^2}$ be an arbitrary
element of $\mathcal{U}_{m_0}$ and let $(i,j)$ be fixed. Let
$(t,x)\in (0,T)\times \R^k$ and $\phi\in \cC^{1,2}([0,T]\times \R^k)$ such
that $\tilde{u}^{ij,*}(t,x)=\phi(t,x)$ and $\tilde{u}^{ij,*}\leq
\phi$. Then
$$
    \begin{array}{ll}
\min\{(\tilde{u}^{ij,*}-L^{ij}[(\tilde{u}^{kl,*})_{(k,l)\in
A^1\times A^2}])(t,x) ;
 \max\{(\tilde{u}^{ij,*}-U^{ij}[(\tilde{u}^{kl,*})_{(k,l)\in A^1\times A^2}])(t,x);\qquad\qquad\qquad\\
\qquad-\partial_t\phi(t,x)-\cL\phi(t,x)-g^{ij}(t,x,(\tilde{u}^{kl,*}(t,x))_{(k,l)\in
A^1\times A^2},\sigma(t,x)^\top
D_x\phi(t,x),I_{ij}(t,x,\phi))\}\}\leq 0.
  \end{array}
   $$
By definition, for any $\klg$, $\tilde{u}^{kl}\leq u^{kl}$ and then
$\tilde{u}^{kl,*}\leq u^{kl,*}$. Using now the monotonicity property ({\bf
A2}), we obtain
$$
    \begin{array}{ll}
\min\{(\tilde{u}^{ij,*}-L^{ij}[({u}^{kl,*})_{(k,l)\in A^1\times
A^2}])(t,x);
 \max\{(\tilde{u}^{ij,*}-U^{ij}[({u}^{kl,*})_{(k,l)\in A^1\times A^2}])(t,x);\\ \qquad\qquad
-\partial_t\phi(t,x)-\cL\phi(t,x)-g^{ij}(t,x,[({u}^{kl,*})_{(k,l)\in \G^{-(i,j)}},\tilde{u}^{ij,*}](t,x)
,\sigma(t,x)^\top D_x\phi(t,x),I_{ij}(t,x,\phi))\}\}\leq 0.
  \end{array}
   $$
This means that $\tilde{u}^{ij}$ is a subsolution of the following
equation: $$ \left\{
    \begin{array}{ll}
\min\{(w-L^{ij}[({u}^{kl,*})_{(k,l)\in A^1\times A^2}])(t,x);
 \max\{(w-U^{ij}[({u}^{kl,*})_{(k,l)\in A^1\times A^2}])(t,x);\\
-\partial_tw(t,x)-\cL w(t,x)-g^{ij}(t,x,[({u}^{kl,*})_{(k,l)\in \G^{-(i,j)}},w](t,x) ,\sigma(t,x)^\top D_xw(t,x),I_{ij}(t,x,w))\}\}=0;\\
w(T,x)=h^{ij}(x).
  \end{array}
    \right.$$
Consequently, by a result by Barles-Imbert (\cite{barlesimbert},
Theorem 2, pp.577),  $u^{ij}$ is a subsolution of $$\left\{
    \begin{array}{ll}
\min\{(w-L^{ij}[({u}^{kl,*})_{(k,l)\in A^1\times A^2}])(t,x);
 \max\{(w-U^{ij}[({u}^{kl,*})_{(k,l)\in A^1\times A^2}])(t,x);\\
-\partial_tw(t,x)-\cL w(t,x)-g^{ij}(t,x,[({u}^{kl,*})_{(k,l)\in \G^{-(i,j)}},w](t,x) ,\sigma(t,x)^\top D_xw(t,x),I_{ij}(t,x,w))\}\}=0;\\
w(T,x)=h^{ij}(x).
  \end{array}
    \right.$$
As $(i,j)$ in $\G$ is arbitrary then $(u^{ij})_{\ij}$ is a subsolution of (\ref{SIPDE}).
 \ms

\no {\bf Step 2:} We will now show, by contradiction, that
$(u^{ij})_{\ij}$ is a supersolution of  (\ref{SIPDE}). First note that
for any $(i,j)\in A^1\times
    A^2$, $\underline{u}^{ij}=\underline{u}^{ij}_*\leq
    \overline{u}^{ij}_*\leq u^{ij}_*\leq
    \overline{u}^{ij,m_0}_*=\overline{u}^{ij,m_0},$ since
    $\overline{u}^{ij,m_0}$ is continuous and $\underline{u}^{ij}$ is
    $lsc$. Therefore, for any $x\in \R^k$, since
    $\underline{u}^{ij}(T,x)=h^{ij}(x)=\overline{u}^{ij,m_0}(T,x)$, it holds, $
    u^{ij}_*(T,x)=h^{ij}(x).$ 
    \ms
    
The rest of the proof is rather classical and can be read e.g. in \cite{Bis} up to some adaptations. However we give it for completeness. So suppose that there exist $\ij$, $(t,x)\in(0,T)\times \R^k$ and a $\cC^{1,2}$-function $\phi$ such that $u^{ij}_*(t,x)=\phi(t,x)$, $u^{ij}_*(s,y)>\phi(s,y)$ in $(0,T)\times \R^k-\{(t,x)\}$ and
 $$
    \begin{array}{ll}
\min\{(\phi-L^{ij}[({u}^{kl}_*)_{(k,l)\in A^1\times
A^2}])(t,x);
 \max\{(\phi-U^{ij}[({u}^{kl}_*)_{(k,l)\in A^1\times A^2}])(t,x);\qquad\qquad\qquad\\
\qquad
-\partial_t\phi(t,x)-\cL\phi(t,x)-g^{ij}(t,x,[({u}^{kl}_*)_{(k,l)\in
\G^{-(i,j)}},\phi](t,x) ,\sigma(t,x)^\top
D_x\phi(t,x),I_{ij}(t,x,\phi))\}\}<0.
  \end{array}
$$
Then by continuity of the equation, continuity and monotonicity of
$g^{ij}$ and lower semi-continuity of  $u^{kl}_*$, there exist two
constants $\eps_1>0$ and $\d_1>0$ such that: $\forall (s,y)\in
B((t,x),{\d_1})$ and $0\leq \eps \leq \eps_1$ we have: \be
\label{eqcontradiction}
    \begin{array}{ll}
\min\{(\phi_\eps-L^{ij}[({u}^{kl}_*)_{(k,l)\in A^1\times
A^2}])(s,y);
 \max\{(\phi_\eps -U^{ij}[({u}^{kl}_*)_{(k,l)\in A^1\times A^2}])(s,y);\qquad\qquad\qquad\\
\quad -\partial_t\phi_\eps (s,y)-\cL\phi_\eps
(s,y)-g^{ij}(s,y,[({u}^{kl}_*)_{(k,l)\in \G^{-(i,j)}},\phi_\eps
](s,y) ,\sigma(s,y)^\top
D_x\phi_\eps(s,y),I_{ij}(s,y,\phi_\eps))\}\}\leq 0
  \end{array}
\ee
where $\phi_\eps=\phi+\eps$. Next since $(t,x)$ is a strict minimum of $u^{ij}_*-\phi$, there are
constants $0<\eps_2$ and $0<\d_2\leq \d_1$ such that $u^{ij}_*-\phi
>\eps_2$ on $\partial B((t,x),{\d_2})$. Now let us set $\eps_3=\min
(\eps_1,\eps_2)$ and let us define $(w^{kl})_{\kl}$ as follows:
 $$
 w^{kl}=u^{kl,*} \mbox{ if } (k,l)\neq (i,j) \mbox{ and }w^{ij}=\left \{\begin{array}{l}
 \max (\phi +\eps_3, u^{ij,*}) \mbox{ on }B((t,x),{\d_2})\subset (0,T)\times \R^k \,\,;\\
 u^{ij,*} \mbox{ elsewhere }.
 \end{array}
 \right.
 $$
 Then $(w^{kl})_{\kl}$ belongs to $\Pi_g$ and is a viscosity subsolution of (\ref{SIPDE}). Indeed, first
 note that for any $\klg$, $w^{kl}$ is usc and $w^{kl}(T,x)=u^{kl,*}(T,x)=h^{kl}(x).$ Next let $(s,y)\in
 (0,T)\times \R^k$. If $(s,y)$ does not belong to $B((t,x),{\d_2})$ then the
 subsolution property stems from the one of $(u^{kl,*})_{\kl}$. Assume now
 that $(s,y)\in B((t,x),{\d_2})$. If $(k,l)\neq (i,j)$, then the
 subsolution property stems from the one of $u^{kl,*}$, in taking
 into account of $w^{ij}\geq u^{ij,*}$, $w^{k_1l_1}=u^{k_1l_1,*}$ if $({k_1,l_1})\neq (i,j)$
 and monotonicity of
 $g^{kl}$. Finally let us deal with the case when $(k,l)=(i,j)$. Let
 $\psi$ be a $\cC^{1,2}$-function such that $w^{ij}(s,y)=\psi(s,y)$
 and $\psi-w^{ij}$ has a strict global  minimum in $(s,y)\in (0,T)\times \R^k$. If
 $w^{ij}(s,y)=u^{ij,*}(s,y)$ then
                $$
    \begin{array}{ll}
\min\{(\psi-L^{ij}[({u}^{kl,*})_{(k,l)\in A^1\times A^2}])(s,y);
 \max\{(\psi-U^{ij}[({u}^{kl,*})_{(k,l)\in A^1\times A^2}])(s,y);\qquad\qquad\qquad\\
\qquad
-\partial_t\psi(t,x)-\cL\psi(s,y)-g^{ij}(s,y,[({u}^{kl,*})_{(k,l)\in
\G^{-(i,j)}},\psi](s,y) ,\sigma(s,y)^\top
D_x\psi(s,y),I_{ij}(s,y,\psi))\}\}\leq 0
  \end{array}
$$
since $w^{ij}\geq u^{ij,*}$ and then $\psi-u^{ij,*}$ has a strict
global  minimum in $(s,y)$, $(\psi-u^{ij,*})(s,y)=0$ and
$(u^{kl,*})_{\kl}$ is a subsolution of system (\ref{SIPDE}). Thus,
  \be \label{eqcontradiction2}
    \begin{array}{ll}
\min\{(\psi-L^{ij}[({w}^{kl,*})_{(k,l)\in A^1\times A^2}])(s,y);
 \max\{(\psi-U^{ij}[({w}^{kl,*})_{(k,l)\in A^1\times A^2}])(s,y);\qquad\qquad\qquad\\
\qquad
-\partial_t\psi(t,x)-\cL\psi(s,y)-g^{ij}(s,y,[({w}^{kl,*})_{(k,l)\in
\G^{-(i,j)}},\psi](s,y) ,\sigma(s,y)^\top
D_x\psi(s,y),I_{ij}(s,y,\psi))\}\}\leq 0.
  \end{array}
\ee
Finally if $w^{ij}(s,y)=\phi(s,y)+\eps_3$ then
$\phi(s,y)+\eps_3=\psi(s,y)$ and $\phi+\eps_3\leq \psi$ on $B((t,x),{\d_2})$. It implies
that
$$
\partial_t \phi(s,y)= \partial_t \psi(s,y),  D_x \phi(s,y)= D_x \psi(s,y)
\mbox{ and }D^2_{xx} \phi(s,y)\leq D^2_{xx}  \psi(s,y)
$$ and by (\ref{eqcontradiction}) we deduce that $w^{ij}$ verifies (\ref{eqcontradiction2}).
Therefore $w^{ij}$ satisfies the subsolution property and
$(w^{kl})_{\kl}$ is a viscosity subsolution of (\ref{SIPDE}). But we
have, $$ w^{ij}_*(t,x)\geq \max\{\phi(t,x)+\eps_3,
u^{ij}_*(t,x)\}=\phi(t,x)+\eps_3= u^{ij}_*(t,x)+\eps_3.$$Thus there
exists $(t_0,x_0)\in (0,T)\times \R^k$ such that
$w^{ij}(t_0,x_0)>u^{ij}(t_0,x_0)$. But this is contradictory to the
definition of $u^{ij}$. Therefore $(u^{kl})_{\kl}$ is a
supersolution of (\ref{SIPDE}) and the proof is complete.

Now, by Corollary (\ref{coruni}), $(^{m_0}u^{ij})_{(i,j)\in
A^1\times A^2}$ (i.e. $(u^{ij})_{(i,j)\in A^1\times A^2}$) does not
depend on $m_0$ since the solution of (\ref{SIPDE}) is the unique.
On the other hand for any $(i,j)$, we have
$$\bar u^{ij}\leq u^{ij}:={}^{m_0}u^{ij} \leq
u^{ij,m_0}$$ and in taking the limit wrt $m_0$ we obtain $\bar u^{ij}= u^{ij}$, for any $(i,j)\in A_1\times A_2$.
\end{proof}
As a by-product of the above construction we have the following
result related to the limit of the increasing scheme:
\begin{thm}The family $(\underbar u^{ij})_{(i,j)\in A^1\times A^2}$ is continuous and of polynomial growth and is the unique viscosity solution in $\Pi_g$ of the max-min problem, i.e.,
for any $\ijg$, \be \label{maxSIPDE}\left\{
    \begin{array}{ll}
\max\{(v^{ij}-U^{ij}[\vec{v}])(t,x);\min\{(v^{ij}-L^{ij}[\vec{v}])(t,x) ; \\
\qq-\partial_t v^{ij}(t,x)-\cL
v^{ij}(t,x)-g^{ij}(t,x,(v^{kl}(t,x))_{(k,l)\in
A^1\times A^2},\sigma(t,x)^\top D_x v^{ij}(t,x),I^{{ij}}(t,x,v^{ij}))\}\}=0\,;\\
     v^{ij}(T,x)=h^{ij}(x).
    \end{array}
    \right.\ee
\end{thm}
 \begin{proof} Actually in considering the opposite of the
 increasing scheme defined in (\ref{increasing}), which becomes a decreasing one, we obtain that
 $(-\underline u^{ij})_{(i,j)\in A^1\times A^2}$ is
 continuous and of polynomial growth and is the unique viscosity
 solution in $\Pi_g$ of the following system: $\forall (i,j)\in A^1\times A^2$,$$\left\{
    \begin{array}{ll}
\min\{(\underbar v^{ij}-\max_{k\in A^2_j}\{\underbar v^{ik}-\bar
g_{kj}\})(t,x); \max\{ (\underbar
v^{ij}-\min_{l\in A^1_i}\{\underbar v^{lj}+\underline g_{il}\} )(t,x) ; \\
-\partial_t \underbar v^{ij}(t,x)-\cL \underbar
v^{ij}(t,x)+g^{ij}(t,x,(-\underbar v^{kl}(t,x))_{(k,l)\in
A^1\times A^2},\sigma(t,x)\tp D_x (-\underbar v^{ij})(t,x),-I_{{ij}}(t,x,\underbar v^{ij}))\}\}=0\,;\\
     \underbar v^{ij}(T,x)=-h^{ij}(x).
    \end{array}
    \right.$$
Using now a result by G.Barles (\cite{barles2}, pp.18) we obtain
that $(\underline u^{ij})_{(i,j)\in A^1\times A^2}$ is the unique
viscosity solution in $\Pi_g$ of (\ref{maxSIPDE}). \end{proof}

\section{Appendix : Alternative definition of the viscosity solution of system
(\ref{SIPDE})} \no The following result  inspired by the work by
Barles-Imbert \cite{barlesimbert} is another definition of the
viscosity solution of system (\ref{SIPDE}). We do not give its proof
since it is an adaptation of the one given in (\cite{zhao},
Proposition 5.2, pp.1656) as the function $q\longmapsto
g^{ij}(t,x,\vec{y},z,q)$ is non-decreasing, $\beta$ is a bounded
function, $\gamma^{ij}$ is non-negative which then imply
$I_{ij}(t,x,\phi)\leq I_{ij}(t,x,\psi)$, $
I_{ij}^{1,\d}(t,x,\phi)\le  I_{ij}^{1,\d}(t,x,\psi)$ and
$I_{ij}^{2,\d}(t,x,\phi)\le I_{ij}^{2,\d}(t,x,\psi)$ for any $\phi
\leq \psi$ such that  $\phi(t,x)=\psi(t,x)=u^{ij}(t,x)$  ($\delta
>0$ and $\ijg$ are fixed). \ms
\begin{propo} \label{defappendix}A function $\vec{u}=(u^{ij}(t,x))_{\ijg}\,\,:[0,T]\times \R^k \rightarrow \R^{m_1\times m_2}$ such that
for any $\ijg$, $u^{ij}\in\Pi_g$ is lsc (resp.
usc) is a viscosity supersolution  (resp. subsolution) of
(\ref{SIPDE}) if:
\ms

\no (i) $ v^{ij}(T,x_0)\geq\,\, (resp. \leq) \,\, h^{ij}(x_0)$, $\forall x_0\in \R^k$ ;
\ms

\no (ii) For any $(t_0,x_0)\in (0,T)\times \R^k$, $\delta \in (0,1)$
and a function $\phi\in \cC^{1,2}([0,T]\times \R^k)$ such that
$u^{ij}(t_0,x_0)=\phi(t_0,x_0)$ and $u^{ij}-\phi$ has a global
minimum (resp. maximum) at $(t_0,x_0)$ on $(0,T)\times B(x_0,\d
K_\b)$ where $K_\b$ is the bound of $\beta$ (see the first
inequality of (A0)-(ii)), we have:
$$ \left\{
    \begin{array}{ll}
\min\{(u^{ij}-L^{ij}[\vec{u}])(t_0,x_0); \max\{(u^{ij}-U^{ij}[\vec{u}])(t_0,x_0);\\
-\partial_t\phi(t_0,x_0)-b(t_0,x_0)^{\top} D_x\phi(t_0,x_0)-\frac{1}{2}\mbox{Tr}(\sigma\sigma^\top(t_0,x_0)D^2_{xx}\phi(t_0,x_0))-I^1_\delta(t_0,x_0,\phi)-I^2_\delta(t_0,x_0,D_x\phi,u^{ij})\\
-g^{ij}(t_0,x_0,(u^{kl}(t_0,x_0))_{(k,l)\in A^1\times
A^2},\sigma(t_0,x_0)^\top
D_x\phi(t_0,x_0),I^{1,\delta}_{ij}(t_0,x_0,\phi)+I^{2,\delta}_{ij}(t_0,x_0,u^{ij}))\}\}\geq
(resp. \leq)\,0.
  \end{array}
    \right.$$
 \end{propo}
 \ms

 \begin{rem} \label{remappendix} In taking $\bar g_{jl}\equiv +\infty$ (resp.
 $\underline g_{ik}\equiv +\infty$) for any $j,l\in A_2$ (resp.
 $i,k\in A_1$) we obtain an alternative definition of the viscosity solution of the system of
 variational inequalities with interconnected lower (resp. upper)
 obstacles.\qed
                 \end{rem}
\ed